\documentclass[reqno,fleqn]{article}

\usepackage{amsthm,amsmath,amstext,amssymb}
\usepackage[all]{xy}
\usepackage{combelow} 
\usepackage{tensor} 
\usepackage{color}
\usepackage[a4paper]{geometry}
\usepackage{hyperref} 

\newtheorem{theorem}{Theorem}[section]
\newtheorem{lem}[theorem]{Lemma}
\newtheorem{prop}[theorem]{Proposition}
\newtheorem{corollary}[theorem]{Corollary}

\theoremstyle{definition}
\newtheorem{defi}[theorem]{Definition}
\newtheorem{example}[theorem]{Example}
\newtheorem{remark}[theorem]{Remark}

\newtheorem{subsec}[theorem]{}

\newcommand{\oG}{\bar{G}}
\newcommand{\g}{\bar{g}}
\newcommand{\h}{\bar{h}}
\renewcommand{\i}{\bar{i}}
\renewcommand{\j}{\bar{j}}
\newcommand{\x}{\bar{x}}
\newcommand{\y}{\bar{y}}
\newcommand{\z}{\bar{z}}
\renewcommand{\t}{\bar{t}}
\renewcommand{\O}{\mathcal{O}}
\newcommand{\K}{\mathcal{K}}
\newcommand{\C}{\mathcal{C}}
\newcommand{\Z}{\mathcal{Z}}
\DeclareFontFamily{OT1}{pzc}{}
\DeclareFontShape{OT1}{pzc}{m}{it}{<->s*[1.10] pzcmi7t}{}
\DeclareMathAlphabet{\mathscr}{OT1}{pzc}{m}{it}
\renewcommand{\k}{\mathscr{k}}
\newcommand{\set}[1]{\left\{#1\right\}}
\newcommand{\para}[1]{\left(#1\right)}
\newcommand{\Hom}[3]{\mathrm{Hom}_{#1}(#2,#3)}

\newcommand{\op}{^\mathrm{op}}
\newcommand{\md}{\textrm{-}\mathrm{mod}}
\newcommand{\Md}{\textrm{-}\mathrm{Mod}}
\newcommand{\triple}[3]{\left(#1,#2,#3\right)}
\newcommand{\Ms}{M^{\ast}}

\title{Character triples and equivalences over a group graded $G$-algebra}
\author{Andrei Marcus {\rm and} Virgilius-Aurelian Minu\cb{t}\u{a}}
\date{}

\begin{document}

\maketitle

{\bf Abstract.} We introduce Morita and Rickard equivalences over a group graded $G$-algebra between block extensions. A consequence of such equivalences is that Sp\"ath's central order relation holds between two corresponding character triples.

\bigskip

{\bf Keywords.} Finite group algebras, block extensions, character triples, group graded algebras, Morita equivalence, Rickard equivalence.

\bigskip

{\bf MSC 2010.} 20C20, 20C05, 16W50, 16S35, 16D90, 18E30.

\renewcommand{\thefootnote}{}
\footnotetext[1]{\center
\begin{tabular}{c}
{\small\em Babe\cb s-Bolyai University} \\
{\small\em Department of Mathematics}\\
{\small\em Str. Mihail Kog\u alniceanu nr. 1}\\
{\small\em  400084 Cluj-Napoca, Romania}\\
{\small\em E-mail: {\tt marcus@math.ubbcluj.ro}} \\
{\small\em E-mail: {\tt minuta.aurelian@math.ubbcluj.ro}}
\end{tabular}}

\section{Introduction}  \label{s:preliminaries}

An important strategy to tackle the local-global conjectures in modular representation theory of finite groups is to try to reduce them to stronger statements about simple groups. Such reduction theorems are obtained by techniques of Clifford theory, and in particular, by using the language of character triples. This paper is motivated by the recent results of Britta Sp\"ath, surveyed in \cite{ch:Spath2017} and \cite{ch:Spath2018}, on inductive conditions for the McKay conjecture and the Alperin weight conjecture. Her reduction theorems involve certain order relations between two character triples, which forces that the two triples have ``the same Clifford theory".

Another motivation comes from the convictions, expressed primarily in the work of Michel Brou\'e (see, for instance, \cite{article:Broue1990} and \cite{article:Broue1992}), that character correspondences with good properties are consequences of categorical equivalences, like Morita equivalences or Rickard equivalences between blocks of group algebras. To explain the link between these two points of view, let us introduce our context.

We consider a finite group $G$,  a $p$-modular system $(\K,\O,\k)$, where $\O$ is a complete discrete valuation ring, $\K$ is the field of fractions of $\O$  and $\k=\O/J(\O)$ is its residue field, together with the assumptions that $\k$ is algebraically closed, and that $\K$ contains all the unity roots of order $|G|$.

Let $N$ be a normal subgroup of $G$, and denote $\oG:=G/N$. Let also  $G'$ be a subgroup of $G$ such that $G=NG'$, and let $N'=G'\cap N$. Let $b$ and $b'$ be $\bar G$-invariant blocks of $\mathcal{O}N$ and $\mathcal{O}N'$, respectively, and consider the strongly $\oG$-graded $\O$-algebras $A=b\mathcal{O}G$ and $A'=b'\mathcal{O}G'$, with identity components $B=b\mathcal{O}N$ and $B'=b'\mathcal{O}N'$. One of the conditions in Sp\"ath's definition \cite[Definition 3.1]{ch:Spath2017}, \cite[Definition 2.7]{ch:Spath2018} of the central order $\le_c$ between two character triples is that $C_G(N)\le G'$ and  that the projective characters associated to the character triples  take the same scalar values on elements of $C_G(N)$. Observe that in this situation, the group algebra $\mathcal{C}=\mathcal{O}C_G(N)$ is a $\bar G$-graded subalgebra of both $A$ and $A'$, and has an obvious $\bar G$-action compatible with the $\bar G$-grading.

We already know that equivalences induced by $\bar G$-graded bimodules preserve many Clifford theoretical invariants (see \cite[Chapter 5]{book:Marcus1999}, \cite{art:Marcus2009} and \cite{article:MM2019})). The relation $\le_c$ between  character triples leads us to the consideration of $\bar G$-graded $(A,A')$-bimodules $\tilde M$ satisfying $m_{\bar g}c={}^{\bar g}\,c m_{\bar g}$ for all $\bar g\in \bar G$, $c\in \mathcal{C}$ and $m_{\bar g}\in \tilde M_{\bar g}$. It turns out that equivalences induced by such bimodules imply the relation $\le_c$ between corresponding character triples.

In Section \ref{s:bimodules} we introduce $\bar G$-graded $(A,A')$-bimodules over a $\bar G$-graded $\bar G$-acted algebra $\mathcal{C}$, and we establish their main properties. The main result of this section is Theorem \ref{th:equivalent_categories}, where we show that the category $A\textrm{-}\mathrm{Gr}/\C{\textrm-}A'$ is equivalent to the category of $\Delta^{\mathcal{C}}$-modules, where $\Delta^{\mathcal{C}}=\bigoplus_{\bar g\in \bar G}A_{\bar g}\otimes_{\mathcal{C}}{A'}^{\mathrm{op}}_{\bar g}$. Actually, there are three naturally isomorphic functors giving this equivalence, and we also prove in Proposition \ref{lemma:hom} that these functors are also compatible with tensor products and with taking homomorphisms. This property is required in Section  \ref{s:Morita}, where we show that $\bar G$-graded Morita equivalences over $\mathcal{C}$ can be induced from certain equivariant Morita equivalences between $B$ and $B'$. We also prove in Theorem \ref{th:butterfly} an analogue of the Butterfly theorem \cite[Theorem 2.16]{ch:Spath2018}, generalizing the main result of \cite{article:MM2019}. In Section \ref{s:scott} we show how to obtain $\bar G$-graded Morita equivalences over $\mathcal{C}$ from the Morita equivalences induced by the Scott module $\operatorname{Sc}(N\times N',\Delta Q)$ of Koshitani and Lassueur \cite{article:KL1},  \cite{article:KL2}. We present several of the results of Sections \ref{s:preliminaries}, \ref{s:Morita} and \ref{s:Rickard} in the context of $\bar G$-graded crossed product, but note that those result can be easily shown to hold for more general strongly $\bar G$-graded algebras $A$ and $A'$.

In order to take advantage of this categorical setting in Section \ref{s:triples} we introduce module triples $(A,B,V)$, where $V$ is a $\bar G$-invariant simple $\mathcal{K}\otimes_{\mathcal{O}}B$-module, and the relation $\ge_c$ between two module triples $(A,B,V)$ and $(A',B',V')$. Our final main result is Theorem \ref{t:eq-triples}, where we prove that if the module triples $(A,B,V)$ and $(A',B',V')$ correspond under a $\bar G$-graded Rickard equivalence over $\mathcal{C}=\mathcal{O}C_G(N)$, then $(A,B,V)\ge_c(A',B',V')$; this, in turn, implies the relation $(G,N,\theta)\ge_c(G'N',\theta')$ between the associated character triples.

For any unexplained concepts and results, our general references are \cite{book:L2018} and \cite{book:Marcus1999}. Finally, note  that our approach to character triples  is different from that of Turull \cite{art:turull2017}. We will discuss the precise relationship in another paper.


\section{Bimodules over a \texorpdfstring{$\bar{G}$}{G}-graded \texorpdfstring{$\bar{G}$}{G}-acted algebra} \label{s:bimodules}


\begin{subsec} Let $A=\bigoplus_{\bar{g}\in\oG} A_{\bar{g}}$  be a  $\oG$-graded $\O$-algebra with the identity component $B:=A_1$. We will assume that $A$ is a crossed product, and  we choose invertible homogeneous elements $u_{\bar{g}}$ in the component $A_{\bar{g}}$, for any $\bar g\in \bar G$. Let also  $A'$ be another crossed product with identity component $B'$, and choose invertible homogeneous elements  $u'_{\bar{g}}\in A'_{\bar{g}}$, for all $\bar{g}\in \oG$.
\end{subsec}

\begin{defi}  An $\mathcal{O}$-algebra $\C$ is a \textit{$\oG$-graded $\oG$-acted $\mathcal{O}$-algebra} if
 \begin{enumerate}
	\item[(1)] $\C$ is $\oG$-graded, and we write $\C=\bigoplus_{\bar{g}\in \bar G}\C_{\bar{g}}$;
	\item[(2)] $\oG$ acts on $\C$ (always on the left in this article);
	\item[(3)] for all $\bar{g},\bar{h}\in \oG$ and for all $c\in \C_{\bar{h}}$ we have  $\tensor*[_{}^{\bar{g}}]{c}{_{}^{}}\in \C_{\tensor*[_{}^{\bar{g}}]{\bar{h}}{_{}^{}}}$.
 \end{enumerate}
We denote the 1-component of $\C$ by $\Z:=\C_1$, which is a $\oG$-acted algebra.
\end{defi}

\begin{example}   The centralizer $C_A(B)$ of $B$ in $A$ is a $\oG$-graded $\oG$-acted algebra, where for all $\bar{h}\in\oG$,
\[C_A(B)_{\bar{h}}=\set{a\in A_{\bar{h}}\mid ab=ba,\,\forall b\in B},\] so $C_A(B)_1=Z(B)$, and the action is given by ${}^{\bar g}c=u_{\bar{g}}cu_{{\bar{g}}^{-1}},$ for all $c\in C_A(B)_{\bar{h}}$ and  $\bar{g},\bar{h}\in \oG$. This definition does not depend on the choice of the elements $u_{\bar{g}}$. This example goes back to  Dade's work \cite{article:Dade1973} on Clifford theory of blocks.
\end{example}

\begin{defi} \label{def:graded_algs_over_C} Let $\C$ be a $\oG$-graded $\oG$-acted $\O$-algebra. We say the $A$ is a {\it $\bar G$-graded $\mathcal{O}$-algebra over $\C$}   if there is a  $\oG$-graded $\oG$-acted algebra homomorphism
\[\zeta:\C\to C_A(B),\] i.e. for any $\bar{h}\in \oG$ and $c\in \C_{\bar{h}}$, we have $\zeta(c)\in C_A(B)_{\bar{h}}$, and for every $\bar{g}\in\oG$, $\zeta(\tensor*[^{\bar{g}}]{c}{})=\tensor*[^{\bar{g}}]{\zeta(c)}{}$.
\end{defi}

\begin{defi} \label{def:graded_bimodules_over_c} Let $A$ and $A'$ be two $\bar G$-graded crossed products over $\mathcal{C}$, with structure maps $\zeta$ and $\zeta'$, respectively.

 a)	We say that $\tilde{M}$ is a $\oG$-\textit{graded $(A,A')$-bimodule over ${\C}$} if:
	\begin{enumerate}
		\item[(1)] $\tilde{M}$ is an $(A,A')$-bimodule;
		\item[(2)] $\tilde{M}$ has a decomposition $\tilde{M}=\bigoplus_{\bar{g}\in\oG}\tilde{M}_{\bar{g}}$ such that $A_{\bar{g}}\tilde{M}_{\bar{x}}A'_{\bar{h}}\subseteq \tilde{M}_{\bar{g}\bar{x}\bar{h}}$, for all $\bar{g}, \bar{x},\bar{h}\in \oG$;
		\item[(3)] $\tilde{m}_{\bar{g}}\cdot c=\tensor*[^{\bar{g}}]{c}{}\cdot \tilde{m}_{\bar{g}}$, for all $c\in \C,\tilde{m}_{\bar{g}}\in\tilde{M}_{\bar{g}},\bar{g}\in \oG$, where $c \cdot \tilde {m} = \zeta(c)\cdot \tilde{m}$ and $\tilde{m}\cdot c=\tilde{m}\cdot \zeta'(c)$, for all $c\in \C,\tilde{m}\in\tilde{M}$.
	\end{enumerate}

b) $\oG$-graded $(A,A')$-bimodules over $\C$ form a category, which we will denote by $A{\textrm-Gr}/\C{\textrm-}A'$, where the morphisms between $\oG$-graded $(A,A')$-bimodules over $\C$ are just homomorphism between $\oG$-graded $(A,A')$-bimodules.
\end{defi}

\begin{remark} Condition (3) of Definition \ref{def:graded_bimodules_over_c} can be replaced by
\begin{enumerate}
\item[$(3')$] $m\cdot c= c\cdot m$, for all $c\in \C$, $m\in \tilde{M}_1$.
\end{enumerate}
Indeed,  it is obvious that (3) implies $(3')$.  Conversely, note that by \cite[Lemma 1.6.3]{book:Marcus1999}, $\tilde{M}\simeq A\otimes_B M$ as $\oG$-graded $(A,A')$-bimodules, where $M=\tilde{M}_1$. Hence, for all $\tilde{m}_{\g}\in \tilde{M}_{\g}$, there exists  $a_{\g}\in A_{\g}$ and $m\in M$ such that $\tilde{m}_{\g}=a_{\g}m$, and  there exists $b\in B$ such that $a_{\g}=u_{\g}b$. Thus we have:
	\[\tilde{m}_{\g}\cdot c=a_{\g}mc=a_{\g}cm=u_{\g}bcm=u_{\g}cbm=u_{\g}cu_{\g}^{-1}u_{\g}bm =\tensor*[^{\bar{g}}]{c}{}a_{\g}m=\tensor*[^{\bar{g}}]{c}{}\cdot \tilde{m}_{\bar{g}}.\]
\end{remark}

\begin{subsec} We regard ${A'}\op$ as a $\oG$-graded algebra with components $(A'{\op})_{\bar{g}}={A'}\op_{\g}=A'_{\bar{g}^{-1}}$, $\forall \bar{g}\in\oG$. We denote by ``$\ast$" the multiplication in ${A'}\op$. We consider the diagonal part of $A\otimes_{\C} {A'}\op$:
\[\Delta^{\C}:=\Delta(A\otimes_{\C} {A'}\op):=\bigoplus_{\bar{g}\in\oG}A_{\bar{g}}\otimes_{\C} A'_{\bar{g}^{-1}}.\]
Note that $\Delta^{\C}$ is well-defined. Indeed, consider $\g\in\oG$ and $a_{\g}\in A_{\g}$, ${a'}\op_{\g}\in {A'}\op_{\g}$; if  $\x,\y\in\oG$ are such that $\x\y=\g$ and $a_{\x}\in A_{\x}$, $c_{\y}\in \C_{\y}$ are such that $a_{\g}=a_{\x}c_{\y}$, then we have
	\[		a_{\g}\otimes_{\C} {a'}\op_{\g}  =  a_{\x}c_{\y}\otimes_{\C} {a'}\op_{\g}=  a_{\x}\otimes_{\C} c_{\y}{a'}\op_{\g};	\]
but $c_{\y}{a'}\op_{\g}\in A'_{\y}{A'}\op_{\g}=A'_{\y}A'_{\g^{-1}}\subseteq A'_{\y\g^{-1}}=A'_{\y(\x\y)^{-1}}=A'_{\x^{-1}}={A'}\op_{\x}$, and $a_{\x}\in A_{\x}$, thus $a_{\x}\otimes_{\C} c_{\y}{a'}\op_{\g}\in A_{\x}\otimes_{\C} {A'}\op_{\x}\subseteq \Delta^{\C}$.
\end{subsec}

\begin{lem} \label{lemma:Delta-Algebra} {\rm1)}	$\Delta^{\C}$ is an $\O$-algebra.

{\rm2)} $A\otimes_{\C} {A'}\op$ is a right $\Delta^{\C}$-module and a $\oG$-graded $(A,A')$-bimodule over $\C$.
\end{lem}

\begin{proof} 1) 	We define the multiplication as follows:
	\[(a_{\g}\otimes_{\C} {{a_{\g}'}}\op)(a_{\h}\otimes_{\C} {{a_{\h}'}}\op):=a_{\g}a_{\h}\otimes_{\C} {a_{\g}'}{\op}\ast {a_{\h}'}\op,\]
for all $a_{\g}\in A_{\g},\,a_{\h}\in A_{\h},{a'}\op_{\g}\in {A'}\op_{\g},\,{a'}\op_{\h}\in {A'}\op_{\h}$, for all $\g,\h\in \oG$.
	We prove that this  is well-defined, i.e. it is $\C$-balanced. Indeed, let
	$\g,\h\in\oG$ and $a_{\g}\in A_{\g}$, $a_{\h}\in A_{\h}$, ${a'}\op_{\g}\in {A'}\op_{\g}$, ${a'}\op_{\h}\in {A'}\op_{\h}$ and if there exists some $\x,\y,\z,\t\in\oG$ such that $\x\y=\g$, $\z\t=\h$ and $a_{\x}\in A_{\x}$, $c_{\y}\in\C_{\y}$, $a_{\z}\in A_{\z}$, $c_{\t}\in\C_{\t}$ such that $a_{\g}=a_{\x}c_{\y}$, $a_{\h}=a_{\z}c_{\t}$, then $a_{\g}\otimes_{\C} {a'}\op_{\g} = a_{\x}\otimes_{\C} c_{\y}{a'}\op_{\g}$ and $a_{\h}\otimes_{\C} {a'}\op_{\h} = a_{\z}\otimes_{\C} c_{\t}{a'}\op_{\h}$. We have:
	\[
	(a_{\g}\otimes_{\C} {a_{\g}'}\op)(a_{\h}\otimes_{\C} {a_{\h}'}\op)=a_{\g}a_{\h}\otimes_{\C} {a_{\g}'}\op\ast {a_{\h}'}\op,
	\]
	and
	\[
\begin{array}{rcl}
	(a_{\x}\otimes_{\C} c_{\y}{a'}\op_{\g})(a_{\z}\otimes_{\C} c_{\t}{a'}\op_{\h}) & = & a_{\x}a_{\z}\otimes_{\C} (c_{\y}{a'}\op_{\g})\ast (c_{\t}{a'}\op_{\h})\\
	& = & a_{\x}a_{\z}\otimes_{\C} c_{\t}{a'}\op_{\h}c_{\y}{a'}\op_{\g};
\end{array}	
	\]
since ${a'}\op_{\h}\in {A'}\op_{\h}=A'_{\h^{-1}}$, there exists some ${b'}\in A'$ such that ${a'}\op_{\h}=u'_{\h^{-1}}{b'}$, hence
\begingroup
\allowdisplaybreaks
	\[
\begin{array}{rcl}
	(a_{\x}\otimes_{\C} c_{\y}{a'}\op_{\g})(a_{\z}\otimes_{\C} c_{\t}{a'}\op_{\h}) & = & a_{\x}a_{\z}\otimes_{\C} c_{\t}u'_{\h^{-1}}{b'}c_{\y}{a'}\op_{\g}\\
	& = & a_{\x}a_{\z}\otimes_{\C} c_{\t}u'_{\h^{-1}}c_{\y}{b'}{a'}\op_{\g}\\
	& = & a_{\x}a_{\z}\otimes_{\C} c_{\t}u'_{\h^{-1}}c_{\y}u'^{-1}_{\h^{-1}}u'_{\h^{-1}}{b'}{a'}\op_{\g}\\
	& = & a_{\x}a_{\z}\otimes_{\C} c_{\t}\tensor*[^{\h^{-1}}]{c}{_{\y}}{a'}\op_{\h}{a'}\op_{\g}\\
	& = & a_{\x}a_{\z}c_{\t}\otimes_{\C} \tensor*[^{\h^{-1}}]{c}{_{\y}}{a'}\op_{\h}{a'}\op_{\g}\\
	& = & a_{\x}a_{\h}\otimes_{\C} \tensor*[^{\h^{-1}}]{c}{_{\y}}{a'}\op_{\h}{a'}\op_{\g}\\
	& = & a_{\x}a_{\h}\tensor*[^{\h^{-1}}]{c}{_{\y}}\otimes_{\C} {a'}\op_{\h}{a'}\op_{\g}\\
	& = & a_{\x}a_{\h}\tensor*[^{\h^{-1}}]{c}{_{\y}}\otimes_{\C} {a'}\op_{\g}\ast {a'}\op_{\h}\\
	& = & a_{\x}a_{\h}u_{\h^{-1}}c_{\y}u^{-1}_{\h^{-1}}\otimes_{\C} {a'}\op_{\g}\ast {a'}\op_{\h};
\end{array}	
	\]
\endgroup
\begingroup
\allowdisplaybreaks
	but
	$a_{\h}\in A_{\h}$ and $u^{-1}_{\h^{-1}}$ is an invertible homogeneous element of $A_{\h}$, thus there exists some $b\in B$ such that $a_{\h}=bu^{-1}_{\h^{-1}}$, hence
	\[
\begin{array}{rcl}
	(a_{\x}\otimes_{\C} c_{\y}{a'}\op_{\g})(a_{\z}\otimes_{\C} c_{\t}{a'}\op_{\h}) & = & a_{\x}bu^{-1}_{\h^{-1}}u_{\h^{-1}}c_{\y}u^{-1}_{h^{-1}}\otimes_{\C} {a'}\op_{\g}\ast {a'}\op_{\h}\\
	& = & a_{\x}bc_{\y}u^{-1}_{\h^{-1}}\otimes_{\C} {a'}\op_{\g}\ast {a'}\op_{\h}\\
	& = & a_{\x}c_{\y}bu^{-1}_{\h^{-1}}\otimes_{\C} {a'}\op_{\g}\ast {a'}\op_{\h}\\
	& = & a_{\g}bu^{-1}_{\h^{-1}}\otimes_{\C} {a'}\op_{\g}\ast {a'}\op_{\h}\\
	& = & a_{\g}a_{\h}\otimes_{\C} {a'}\op_{\g}\ast {a'}\op_{\h}.
\end{array}	
	\]
\endgroup
	Furthermore,  $\Delta^{\C}$ is an $\O$-algebra, because the map
	\[\varphi:C_1\to Z(\Delta^{\C}),\,\varphi(c):=\zeta(c)\otimes_{\C} 1 = 1\otimes_{\C} \zeta'(c),\]
 is a ring homomorphism. Indeed, we check here only the multiplicative property; if $c,c'\in \C$, then
	\[\begin{array}{rcl}
	\varphi(c)\varphi(c')&=&(\zeta(c)\otimes_{\C} 1)(\zeta(c')\otimes_{\C} 1)\\
	&=&\zeta(c)\zeta(c')\otimes_{\C} 1\\
	&=&\varphi(cc').\\
\end{array}		\]

2) In order to prove that $A\otimes_{\C} {A'}\op$ is a right $\Delta^{\C}$-module, we will define the following scalar multiplication, for every $\g,\bar{x},\bar{y}\in\oG$, $a_{\g}\in A_{\g},\,a_{\bar{x}}\in A_{\bar{x}},\,{a'}\op_{\g}\in {A'}\op_{\g},\,{a'}\op_{\bar{y}}\in {A'}\op_{\bar{y}}$:
\[
(a_{\bar{x}}\otimes_{\C} {a'}\op_{\bar{y}})(a_{\bar{g}}\otimes_{\C} {a'}\op_{\bar{g}}) := a_{\bar{x}} a_{\bar{g}}\otimes_{\C} {a'}\op_{\bar{y}} \ast {a'}\op_{\bar{g}},
\]
and we check that it is well defined. Firstly, we check that the operation is $\C$-balanced in the first factor. Let $\g,\bar{x},\bar{y}\in\oG$, $a_{\g}\in A_{\g},\,a_{\bar{x}}\in A_{\bar{x}},\,{a'}\op_{\g}\in {A'}\op_{\g},\,{a'}\op_{\bar{y}}\in {A'}\op_{\bar{y}},\,c\in \C$. On one hand we have:
	\[
	\begin{array}{rcl}
		(a_{\bar{x}}c\otimes_{\C} {a'}\op_{\bar{y}})(a_{\bar{g}}\otimes_{\C} {a'}\op_{\bar{g}}) & = & a_{\bar{x}}c a_{\bar{g}}\otimes_{\C} {a'}\op_{\bar{y}} \ast {a'}\op_{\bar{g}}\\
		& = & a_{\bar{x}}c a_{\bar{g}}\otimes_{\C} {a'}\op_{\bar{g}} {a'}\op_{\bar{y}}, \\
	\end{array}		
	\]
	and on the other hand we have:
	\[
	\begin{array}{rcl}
		(a_{\bar{x}}\otimes_{\C} c{a'}\op_{\bar{y}})(a_{\bar{g}}\otimes_{\C} {a'}\op_{\bar{g}}) & = & a_{\bar{x}} a_{\bar{g}}\otimes_{\C} (c {a'}\op_{\bar{y}}) \ast {a'}\op_{\bar{g}}\\
		& = & a_{\bar{x}} a_{\bar{g}}\otimes_{\C} {a'}\op_{\bar{g}}( c {a'}\op_{\bar{y}}) \\
		& = & a_{\bar{x}} a_{\bar{g}}\otimes_{\C} ({a'}\op_{\bar{g}} c ){a'}\op_{\bar{y}} \\
	\end{array}		
	\]
	and following a similar argument as in the proof of Lemma \ref{lemma:Delta-Algebra} 1), we have that:
	\[
	\begin{array}{rcl}
		(a_{\bar{x}}\otimes_{\C} c{a'}\op_{\bar{y}})(a_{\bar{g}}\otimes_{\C} {a'}\op_{\bar{g}}) & = & a_{\bar{x}} a_{\bar{g}}\otimes_{\C} ({a'}\op_{\bar{g}} c ){a'}\op_{\bar{y}} \\
		& = & a_{\bar{x}} a_{\bar{g}}\otimes_{\C} \tensor*[^{\g^{-1}}]{c}{} {a'}\op_{\bar{g}}  {a'}\op_{\bar{y}} \\
		& = & a_{\bar{x}} a_{\bar{g}} \tensor*[^{\g^{-1}}]{c}{} \otimes_{\C}  {a'}\op_{\bar{g}}  {a'}\op_{\bar{y}} \\
	\end{array}	\]
	which, in a similar way, is equal to $a_{\bar{x}}c a_{\bar{g}}\otimes_{\C} {a'}\op_{\bar{g}} {a'}\op_{\bar{y}}$. Secondly, we check that the operation is $\C$-balanced in the second factor. Consider $\g,\z,\t\in\oG$ and $a_{\g}\in A_{\g}$, $a_{\z}\in A_{\z}$, ${a'}\op_{\g}\in {A'}\op_{\g}$, ${a'}\op_{\t}\in {A'}\op_{\t}$, $c\in \C$ and some $\x,\y\in\oG$ such that $\x\y=\g$ and $a_{\x}\in A_{\x}$, $c_{\y}\in\C_{\y}$ such that $a_{\g}=a_{\x}c_{\y}$ therefore $a_{\g}\otimes_{\C} {a'}\op_{\g} = a_{\x}\otimes_{\C} c_{\y}{a'}\op_{\g}$. We have:
	\[
\begin{array}{rcl}
	(a_{\z}\otimes_{\C} {a'}\op_{\t})(a_{\g}\otimes_{\C} {a'}\op_{\g})=a_{\z}a_{\g}\otimes_{\C} {a'}\op_{\t}\ast {a'}\op_{\g},
\end{array}	
	\]
and
	\[
\begin{array}{rcl}
	(a_{\z}\otimes_{\C} {a'}\op_{\t})(a_{\x}\otimes_{\C} c_{\y}{a'}\op_{\g}) & = & a_{\z}a_{\x}\otimes_{\C} {a'}\op_{\t}\ast (c_{\y}{a'}\op_{\g})\\
	& = & a_{\z}a_{\x}\otimes_{\C} c_{\y}{a'}\op_{\g}{a'}\op_{\t}\\
	& = & a_{\z}a_{\x}c_{\y}\otimes_{\C} {a'}\op_{\g}{a'}\op_{\t}\\
	& = & a_{\z}a_{\x}c_{\y}\otimes_{\C} {a'}\op_{\t}\ast {a'}\op_{\g}\\
	& = & a_{\z}a_{\g}\otimes_{\C} {a'}\op_{\t}\ast {a'}\op_{\g}.
\end{array}	
	\]
	The remaining axioms are easy to verify, hence $A\otimes_{\C} {A'}\op$ is a right $\Delta^{\C}$-module.

	It remains to prove that $A\otimes_{\C} {A'}\op$ is an $\oG$-graded $(A,A')$-bimodule over $\C$. First,  we prove that $A\otimes_{\C} {A'}\op$ is  a left  $\para{A\otimes_{\O} {A'}\op}$-module. We define the multiplication as follows: for every $\g,\,\h,\,\i,\,\j\in\oG$ and for every $a_{\g}\in A_{\g}$, $a_{\i}\in A_{\i}$, ${a'}\op_{\h}\in {A'}\op_{\h}$, ${a'}\op_{\j}\in {A'}\op_{\j}$ we have
	\[\para{a_{\g}\otimes_{\O}{a'}\op_{\h}}\para{a_{\i}\otimes_{\C} {a'}\op_{\j}}:= a_{\g}a_{\i}\otimes_{\C} {a'}\op_{\h}\ast {a'}\op_{\j}.\]
	This is clearly well-defined, and  we will only verify the multiplicative axiom. Consider $\g,\,\h,\,\i,\,\j,\,\x,\,\y\in\oG$ and $a_{\g}\in A_{\g}$, $a_{\i}\in A_{\i}$, $a_{\x}\in A_{\x}$, ${a'}\op_{\h}\in {A'}\op_{\h}$, ${a'}\op_{\j}\in {A'}\op_{\j}$, ${a'}\op_{\y}\in {A'}\op_{\y}$. On one hand we have:
	\begin{align*}
	\para{a_{\g}\otimes_{\O}{a'}\op_{\h}}&\para{\para{a_{\x}\otimes_{\O}{a'}\op_{\y}}\para{a_{\i}\otimes_{\C}{a'}\op_{\j}}} \\
	=&  \para{a_{\g}\otimes_{\O}{a'}\op_{\h}}\para{a_{\x}a_{\i}\otimes_{\C}{a'}\op_{\y}\ast {a'}\op_{\j}}\\
	=& a_{\g}a_{\x}a_{\i}\otimes_{\C} {a'}\op_{\h}\ast {a'}\op_{\y}\ast {a'}\op_{\j}
	\end{align*}		
and on the other hand we have:

	\begingroup
	\allowdisplaybreaks
	\begin{align*}
	\left(\para{a_{\g}\otimes_{\O}{a'}\op_{\h}}\para{a_{\x}\otimes_{\O}{a'}\op_{\y}}\right)&\para{a_{\i}\otimes_{\C}{a'}\op_{\j}} \\
	=&  \para{a_{\g}a_{\x}\otimes_{\O}{a'}\op_{\h}\ast {a'}\op_{\y}}\para{a_{\i}\otimes_{\C}{a'}\op_{\j}}\\
	=& a_{\g}a_{\x}a_{\i}\otimes_{\C} {a'}\op_{\h}\ast {a'}\op_{\y}\ast {a'}\op_{\j}
	\end{align*}
	\endgroup
	Second, it is clear that $A\otimes_{\C} {A'}\op$ an $\oG$-graded $(A,A')$-bimodule with the $\bar x$-component defined by:
	\[(A\otimes_{\C} {A'}\op)_{\bar x}=\set{\sum_{\bar g\bar g'^{-1}=\bar {x}}{a_{\bar g}\otimes_{\C}{a'}\op_{\bar g'}}\mid a_{\bar g}\in A_{\bar g},\,{a'}\op_{\bar g'}\in {A'}\op_{\bar g'}}.\]
We check that for any $\g,\bar x,\h\in\oG$ we have:
	\[A_{\g}(A\otimes_{\C} {A'}\op)_{\bar x}A'_{\h}\subseteq (A\otimes_{\C} {A'}\op)_{\g\bar x\h}.\]
	Indeed, let $\g,\bar x,\h,\bar i,\bar j\in\oG$ such that $\bar i\bar j^{-1}=\bar x$ and $a_{\bar g}\in A_{\bar {g}}$, $a_{\bar i}\in A_{\bar {i}}$, ${a'}\op_{\bar j}\in {A'}\op_{\bar {j}}$ and $a'_{\bar h}\in A'_{\bar {h}}$. We have
	\[a_{\g}(a_{\bar i}\otimes_{\C}  {a'}\op_{\bar j})a'_{\h}=a_{\g}a_{\bar i}\otimes_{\C} {a'}\op_{\h^{-1}} \ast {a'}\op_{\bar j} ,\]
	where $a_{\g}a_{\bar i}\in A_{\g}A_{\i}\subseteq A_{\g\bar i}$, ${a'}\op_{\h^{-1}} \ast {a'}\op_{\bar j} \in {A'}\op_{\h^{-1}}{A'}\op_{\j}\subseteq {A'}\op_{\h^{-1}\j}$ and $(\g\bar i)(\h^{-1}\j)^{-1} =\g\bar i \j^{-1}\h=\g\bar x\h$, hence $a_{\g}a_{\bar i}\otimes_{\C} {a'}\op_{\h^{-1}} \ast {a'}\op_{\bar j}\in(A\otimes_{\C} {A'}\op)_{\g\bar x\h} $.
	
 Finally, we show that for any $a_{\g}\in A_{\g}$, ${a'}\op_{\h}\in {A'}\op_{\h}$ and $c\in \C$, we have:
	\[(a_{\g}\otimes_{\C} {a'}\op_{\h})c=\tensor[^{\g\h^{-1}}]{c}{}(a_{\g}\otimes_{\C} {a'}\op_{\h}).\]
	Indeed,  we have
	\[		\begin{array}{rcl}
			(a_{\g}\otimes_{\C} {a'}\op_{\h})c & = & a_{\g}\otimes_{\C} c*{a'}\op_{\h}\\
			& = & a_{\g}\otimes_{\C} {a'}\op_{\h} c\\
			& = & a_{\g}\otimes_{\C} \tensor[^{\h^{-1}}]{c}{}{a'}\op_{\h} \\
			& = & a_{\g} \tensor[^{\h^{-1}}]{c}{}\otimes_{\C} {a'}\op_{\h} \\
			& = &  \tensor[^{\g\h^{-1}}]{c}{} a_{\g}\otimes_{\C} {a'}\op_{\h} \\
			& = & \tensor[^{\g\h^{-1}}]{c}{}(a_{\g}\otimes_{\C} {a'}\op_{\h}), \\
		\end{array}	
	\]
hence the claim is proved.
\end{proof}

\begin{remark} \label{remark:o_z_c}
\begingroup
\allowdisplaybreaks
If $M$ is a $\Delta^{\C}$-module, we intend to prove that $(A\otimes_{\C} {A'}\op)\otimes_{\Delta^{\C}} M$ is a $\oG$-graded $(A,A')$-bimodule over $\C$. This is suggested by the following observation. Consider the algebras
\[\Delta^{\O}:=\bigoplus_{\g\in\oG}A_{\g}\otimes_{\O}{A'}\op_{\g},\]
which is a $\oG$-graded $\O$-algebra \cite[1.6.1]{book:Marcus1999} and
\[\Delta^{\Z}:=\bigoplus_{\g\in\oG}A_{\g}\otimes_{\Z}{A'}\op_{\g},\]
which is a $\oG$-graded $\O$-algebra of $A$. 
We clearly have the following surjective algebra homomorphism:
\[
\Delta^{\O}\longrightarrow \Delta^{\Z} \longrightarrow \Delta^{\C}.
\]
Thus, we have the following homomorphism
\[
(A\otimes_{\O} {A'}\op)\otimes_{\Delta^{\O}} M \longrightarrow (A\otimes_{\Z} {A'}\op)\otimes_{\Delta^{\Z}} M \longrightarrow (A\otimes_{\C} {A'}\op)\otimes_{\Delta^{\C}} M,
\]
 of abelian groups, hence we obtain the $(A,A')$-bimodule structure on $(A\otimes_{\C} {A'}\op)\otimes_{\Delta^{\C}} M$.
\endgroup
\end{remark}

\begin{theorem} \label{th:equivalent_categories}
The category of $\Delta^{\C}$-modules and the category of $\oG$-graded $(A,A')$-bimodules over $\C$ are equivalent:
\[\xymatrix@C+=4cm{
\Delta^{\C}\Md \ar@<+.5ex>@{->}[r]^{\left(A\otimes_{\C} {A'}\op\right)\otimes_{\Delta^{\C}}-} & \ar@<+.5ex>@{->}[l]^{\left(-\right)_1} A\textrm{-}\mathrm{Gr}/\C{\textrm-}A'
}\]
\end{theorem}

\begin{proof} \textit{\textbf{Step 1.}} Let $M$ be a $\Delta^{\C}$-module. We  prove that $(A\otimes_{\C} {A'}\op)\otimes_{\Delta^{\C}} M$ is a $\oG$-graded $(A,A')$-bimodule over $\C$.  We define on $(A\otimes_{\C} {A'}\op)\otimes_{\Delta^{\C}} M$ a left $A$-module structure, as follows:
\[a_{\g}\para{\para{a_{\h}\otimes_{\C} {a'}\op_{\bar{k}}}\otimes_{\Delta^{\C}}m}:=\para{a_{\g}a_{\h}\otimes_{\C} {a'}\op_{\bar{k}}}\otimes_{\Delta^{\C}}m;\]
for all $\g,\h,\bar{k}\in \oG$, $a_{\g}\in A_{\g}$, $a_{\h}\in A_{\h}$, ${a'}\op_{\bar{k}}\in {A'}\op_{\bar{k}}$, $m\in M$. We also give a right  $A'$-module structure, by defining
\[\para{\para{a_{\h}\otimes_{\C} {a'}\op_{\bar{k}}}\otimes_{\Delta^{\C}}m}a'_{\g}:=\para{a_{\h}\otimes_{\C} {a'}\op_{\g^{-1}}  \ast {a'}\op_{\bar{k}}}\otimes_{\Delta^{\C}}m,\]
for all $\g,\h,\bar{k}\in \oG$, $a'_{\g}\in A'_{\g},a_{\h}\in A_{\h}$, ${a'}\op_{\bar{k}}\in {A'}\op_{\bar{k}}$, $m\in M$.
We verify the $\Delta^{\C}$-balance property. Let $\g,\h,\bar{k},\bar{x} \in \oG$, $ a_{\g}\in A_{\g},a_{\h}\in A_{\h},a_{\bar{x}}\in A_{\bar{x}},{a'}\op_{\bar{k}}\in {A'}\op_{\bar{k}},{a'}\op_{\bar{x}}\in {A'}\op_{\bar{x}},a'_{\g}\in A'_{\g},m\in M$. We have:
\begin{align*}
a_{\g}\para{\para{a_{\h}\otimes_{\C} {a'}\op_{\bar{k}}}\para{a_{\bar{x}}\otimes_{\C} {a'}\op_{\bar{x}}}\otimes_{\Delta^{\C}}m} & =  a_{\g}\para{\para{a_{\h}a_{\bar{x}}\otimes_{\C} {a'}\op_{\bar{k}}\ast {a'}\op_{\bar{x}}}\otimes_{\Delta^{\C}}m}\\
 &=  \para{a_{\g}a_{\h}a_{\bar{x}}\otimes_{\C} {a'}\op_{\bar{k}}\ast {a'}\op_{\bar{x}}}\otimes_{\Delta^{\C}}m
\end{align*}
and
\begin{align*}
a_{\g}\para{\para{a_{\h}\otimes_{\C} {a'}\op_{\bar{k}}}\otimes_{\Delta^{\C}} \para{a_{\bar{x}}\otimes_{\C} {a'}\op_{\bar{x}}}m} & = \para{a_{\g}a_{\h}\otimes_{\C} {a'}\op_{\bar{k}}}\otimes_{\Delta^{\C}} \para{a_{\bar{x}}\otimes_{\C} {a'}\op_{\bar{x}}}m\\
& = \para{a_{\g}a_{\h}\otimes_{\C} {a'}\op_{\bar{k}}}\para{a_{\bar{x}}\otimes_{\C} {a'}\op_{\bar{x}}}\otimes_{\Delta^{\C}} m\\
 &=  \para{a_{\g}a_{\h}a_{\bar{x}}\otimes_{\C} {a'}\op_{\bar{k}}\ast {a'}\op_{\bar{x}}}\otimes_{\Delta^{\C}}m.
\end{align*}
Moreover,
\begingroup
\allowdisplaybreaks
\begin{align*}
	\para{\para{\para{a_{\h}\otimes_{\C} {a'}\op_{\bar{k}}}\para{a_{\bar{x}}\otimes_{\C} {a'}\op_{\bar{x}}}}\otimes_{\Delta^{\C}}m}a'_{\g} &= \para{\para{\para{a_{\h}a_{\bar{x}}\otimes_{\C} {a'}\op_{\bar{k}}\ast {a'}\op_{\bar{x}}}}\otimes_{\Delta^{\C}}m}a'_{\g}\\
	&= \para{\para{a_{\h}a_{\bar{x}}\otimes_{\C} {a'}\op_{\g^{-1}}\ast {a'}\op_{\bar{k}}\ast {a'}\op_{\bar{x}}}}\otimes_{\Delta^{\C}}m
\end{align*}
\endgroup
and
\begin{align*}
	\para{\para{a_{\h}\otimes_{\C} {a'}\op_{\bar{k}}}\otimes_{\Delta^{\C}}\para{a_{\bar{x}}\otimes_{\C} {a'}\op_{\bar{x}}}m}a'_{\g}
	&= \para{a_{\h}\otimes_{\C} {a'}\op_{\g^{-1}}\ast {a'}\op_{\bar{k}}}\otimes_{\Delta^{\C}}\para{a_{\bar{x}}\otimes_{\C} {a'}\op_{\bar{x}}}m \\
	&= \para{a_{\h}\otimes_{\C} {a'}\op_{\g^{-1}}\ast {a'}\op_{\bar{k}}}\para{a_{\bar{x}}\otimes_{\C} {a'}\op_{\bar{x}}}\otimes_{\Delta^{\C}}m \\
	&=\para{\para{a_{\h}a_{\bar{x}}\otimes_{\C} {a'}\op_{\g^{-1}}\ast {a'}\op_{\bar{k}}\ast {a'}\op_{\bar{x}}}}\otimes_{\Delta^{\C}}m.
\end{align*}
We verify that $(A\otimes_{\C} {A'}\op)\otimes_{\Delta^{\C}} M$ is indeed a left $A$-module and a right $A'$-module. Here we only verify the multiplicative axiom. Let $\g,\h,\bar{k},\bar{x} \in \oG$, $ a_{\g}\in A_{\g},a_{\h}\in A_{\h},a_{\bar{x}}\in A_{\bar{x}}$, ${a'}\op_{\bar{k}}\in {A'}\op_{\bar{k}}$, ${a'}\op_{\bar{x}}\in {A'}\op_{\bar{x}}$, $a'_{\g}\in A'_{\g}$, $m\in M$. We have:
\begin{align*}
\para{a_{\g}a_{\bar{x}}}\para{\para{a_{\h}\otimes_{\C} {a'}\op_{\bar{k}}}\otimes_{\Delta^{\C}}m} & = \para{a_{\g}a_{\bar{x}}a_{\h}\otimes_{\C} {a'}\op_{\bar{k}}}\otimes_{\Delta^{\C}}m
\end{align*}
and
\begingroup
\allowdisplaybreaks
\begin{align*}
a_{\g}\para{a_{\bar{x}}\para{\para{a_{\h}\otimes_{\C} {a'}\op_{\bar{k}}}\otimes_{\Delta^{\C}}m}} & = a_{\g}\para{\para{a_{\bar{x}}a_{\h}\otimes_{\C} {a'}\op_{\bar{k}}}\otimes_{\Delta^{\C}}m}\\
& = \para{a_{\g}a_{\bar{x}}a_{\h}\otimes_{\C} {a'}\op_{\bar{k}}}\otimes_{\Delta^{\C}}m.
\end{align*}
\endgroup
Furthermore,
\begin{align*}
\para{\para{a_{\h}\otimes_{\C} {a'}\op_{\bar{k}}}\otimes_{\Delta^{\C}}m}\para{a'_{\g}a'_{\bar{x}}} & = \para{\para{a_{\h}\otimes_{\C} \para{a'_{\g}a'_{\bar{x}}}\op \ast {a'}\op_{\bar{k}}}\otimes_{\Delta^{\C}}m}\\
& = \para{\para{a_{\h}\otimes_{\C} {a'}\op_{\bar{x}^{-1}}\ast {a'}\op_{\g^{-1}} \ast {a'}\op_{\bar{k}}}\otimes_{\Delta^{\C}}m}
\end{align*}
and
\begin{align*}
\para{\para{\para{a_{\h}\otimes_{\C} {a'}\op_{\bar{k}}}\otimes_{\Delta^{\C}}m}a'_{\g}}a'_{\bar{x}} & = \para{\para{a_{\h}\otimes_{\C} {a'}\op_{\g^{-1}}\ast {a'}\op_{\bar{k}}}\otimes_{\Delta^{\C}}m}a'_{\bar{x}}\\
& = \para{\para{a_{\h}\otimes_{\C} {a'}\op_{\bar{x}^{-1}}\ast {a'}\op_{\g^{-1}} \ast {a'}\op_{\bar{k}}}\otimes_{\Delta^{\C}}m}.
\end{align*}
We verify that $(A\otimes_{\C} {A'}\op)\otimes_{\Delta^{\C}} M$ is indeed an $(A,A')$-bimodule; more exactly, we  verify that
\[\para{a_{\g}\para{\para{a_{\h}\otimes_{\C} {a'}\op_{\bar{k}}}\otimes_{\Delta^{\C}}m}}a'_{\bar{x}}=a_{\g}\para{\para{\para{a_{\h}\otimes_{\C} {a'}\op_{\bar{k}}}\otimes_{\Delta^{\C}}m}a'_{\bar{x}}},\]
for all $\g,\h,\bar{k},\bar{x} \in \oG$, $ a_{\g}\in A_{\g},a_{\h}\in A_{\h},{a'}\op_{\bar{k}}\in {A'}\op_{\bar{k}},a'_{\g}\in A'_{\g},a'_{\bar{x}}\in A'_{\bar{x}},m\in M$. Indeed, we have:
\begin{align*}
\para{a_{\g}\para{\para{a_{\h}\otimes_{\C} {a'}\op_{\bar{k}}}\otimes_{\Delta^{\C}}m}}a'_{\bar{x}} & = \para{\para{a_{\g}a_{\h}\otimes_{\C} {a'}\op_{\bar{k}}}\otimes_{\Delta^{\C}}m}a'_{\bar{x}}\\
& = \para{a_{\g}a_{\h}\otimes_{\C} {a'}\op_{\bar{x}^{-1}}\ast {a'}\op_{\bar{k}}}\otimes_{\Delta^{\C}}m
\end{align*}
and
\begingroup
\allowdisplaybreaks
\begin{align*}
a_{\g}\para{\para{\para{a_{\h}\otimes_{\C} {a'}\op_{\bar{k}}}\otimes_{\Delta^{\C}}m}a'_{\bar{x}}}&=a_{\g}\para{\para{a_{\h}\otimes_{\C} {a'}\op_{\bar{x}^{-1}} \ast {a'}\op_{\bar{k}}}\otimes_{\Delta^{\C}}m}\\
&=\para{a_{\g}a_{\h}\otimes_{\C} {a'}\op_{\bar{x}^{-1}} \ast {a'}\op_{\bar{k}}}\otimes_{\Delta^{\C}}m.
\end{align*}
\endgroup
 We define the $x$-component  of $(A\otimes_{\C} {A'}\op)\otimes_{\Delta^{\C}} M$, for every $x\in\oG$, by
\[\para{(A\otimes_{\C} {A'}\op)\otimes_{\Delta^{\C}} M}_x:=(A\otimes_{\C} {A'}\op)_x\otimes_{\Delta^{\C}} M.\]
We verify that  for all $\g,x,\h\in\oG$
\[A_{\g}\para{(A\otimes_{\C} {A'}\op)\otimes_{\Delta^{\C}} M}_x A'_{\h}\subseteq \para{(A\otimes_{\C} {A'}\op)\otimes_{\Delta^{\C}} M}_{\g x\h}.\]
Indeed, for every $a_{\g}\in A_{\g},a_t\in A_t, m\in M, a'_{t'}\in A'_{t'}, a'_{\h}\in A'_{\h}$, with $t,t'\in \oG$ such that $tt'^{-1}=x$, we have:
\[\begin{array}{rcl}
a_{\g}\para{\para{a_t\otimes_{\C} {a'}\op_{t'}}\otimes_{\Delta^{\C}}m}a'_{\h} & = & \para{\para{a_{\g}a_t\otimes_{\C} {a'}\op_{t'}}\otimes_{\Delta^{\C}}m}a'_{\h}\\
& = & \para{a_{\g}a_t\otimes_{\C} {a'}\op_{\h^{-1}}\ast {a'}\op_{t'}}\otimes_{\Delta^{\C}}m\\
& \in & (A\otimes_{\C} {A'}\op)_{\g t(\h^{-1}t')^{-1}}\otimes_{\Delta^{\C}} M\\
& = & (A\otimes_{\C} {A'}\op)_{\g tt'^{-1}\h}\otimes_{\Delta^{\C}} M\\
& = & (A\otimes_{\C} {A'}\op)_{\g x\h}\otimes_{\Delta^{\C}} M\\
& = & \para{(A\otimes_{\C} {A'}\op)\otimes_{\Delta^{\C}} M}_{\g x\h}.
\end{array}
\]
We verify that $(A\otimes_{\C} {A'}\op)\otimes_{\Delta^{\C}} M$ is an $(A,A')$-bimodule over $\C$.  Indeed, for every $\g,t,t',\h\in \oG$ such that $tt'^{-1}=\g$ and for every $a_t\in A_t, m\in M, a'_{t'}\in A'_{t'},c_{\h}\in\C_h$ we have:
\[
\begin{array}{rcl}
\para{\para{a_t\otimes_{\C} {a'}\op_{t'}}\otimes_{\Delta^{\C}}m}c_{\h} & = & \para{a_t\otimes_{\C} c\op_{\h^{-1}} \ast {a'}\op_{t'}}\otimes_{\Delta^{\C}}m\\
\end{array}
\]
and
\[
\begin{array}{rcl}
\tensor*[^{\g}]{c}{_{\h}}\para{\para{a_t\otimes_{\C} {a'}\op_{t'}}\otimes_{\Delta^{\C}}m} & = & \para{\tensor*[^{\g}]{c}{_{\h}}a_t\otimes_{\C} {a'}\op_{t'}}\otimes_{\Delta^{\C}}m\\
& = & \para{u_{\g}c_{\h}u^{-1}_{\g} a_t\otimes_{\C} {a'}\op_{t'}}\otimes_{\Delta^{\C}}m.
\end{array}
\]
Now consider some invertible homogeneous elements $u_{t'^{-1}}$ and $u'_{t'^{-1}}$ in the components $A_{t'^{-1}}$ and $A'_{t'^{-1}}$, respectively. We have:
\[
\begin{array}{rcl}
\tensor*[^{\g}]{c}{_{\h}}\para{\para{a_t\otimes_{\C} {a'}\op_{t'}}\otimes_{\Delta^{\C}}m} & = & \para{u_{\g}c_{\h}u^{-1}_{\g} a_t\otimes_{\C} {a'}\op_{t'}}\otimes_{\Delta^{\C}}m\\
& = & \para{u_{\g}c_{\h}u^{-1}_{\g} a_tu_{t'^{-1}}u^{-1}_{t'^{-1}}\otimes_{\C} {a'}\op_{t'}}\otimes_{\Delta^{\C}}m;\\
\end{array}
\]
but $a_tu_{t'^{-1}}\in A_tA_{t'^{-1}}\subseteq A_{tt'^{-1}}=A_{\g}$, so $u^{-1}_{\g} a_tu_{t'^{-1}}\in B$, but $c_{\h}\in C_A(B)$, thus
\[
\begin{array}{rcl}
\tensor*[^{\g}]{c}{_{\h}}\para{\para{a_t\otimes_{\C} {a'}\op_{t'}}\otimes_{\Delta^{\C}}m} & = & \para{u_{\g}u^{-1}_{\g} a_tu_{t'^{-1}}c_{\h}u^{-1}_{t'^{-1}}\otimes_{\C} {a'}\op_{t'}}\otimes_{\Delta^{\C}}m\\
& = & \para{a_tu_{t'^{-1}}c_{\h}u^{-1}_{t'^{-1}}\otimes_{\C} {a'}\op_{t'}}\otimes_{\Delta^{\C}}m\\
& = & \para{a_t \tensor*[^{t'^{-1}}]{c}{_{\h}}\otimes_{\C} {a'}\op_{t'}}\otimes_{\Delta^{\C}}m;\\
\end{array}
\]
but since $\tensor*[^{t'^{-1}}]{c}{_{\h}}\in \C$, we have
\begingroup
\allowdisplaybreaks
\[
\begin{array}{rcl}
\tensor*[^{\g}]{c}{_{\h}}\para{\para{a_t\otimes_{\C} {a'}\op_{t'}}\otimes_{\Delta^{\C}}m} & = & \para{a_t \otimes_{\C} \tensor*[^{t'^{-1}}]{c}{_{\h}}{a'}\op_{t'}}\otimes_{\Delta^{\C}}m\\
& = & \para{a_t \otimes_{\C} u'_{t'^{-1}}c_{\h}u'^{-1}_{t'^{-1}}{a'}\op_{t'}}\otimes_{\Delta^{\C}}m;\\
\end{array}
\]
\endgroup
by using that $u'^{-1}_{t'^{-1}}{a'}\op_{t'}\in B'$ and $c_{\h}\in C_{A'}(B')$, we get
\[
\begin{array}{rcl}
\tensor*[^{\g}]{c}{_{\h}}\para{\para{a_t\otimes_{\C} {a'}\op_{t'}}\otimes_{\Delta^{\C}}m} & = & \para{a_t \otimes_{\C} u'_{t'^{-1}}u'^{-1}_{t'^{-1}}{a'}\op_{t'}c_{\h}}\otimes_{\Delta^{\C}}m\\
& = & \para{a_t \otimes_{\C} {a'}\op_{t'}c_{\h}}\otimes_{\Delta^{\C}}m\\
& = & \para{a_t \otimes_{\C} c\op_{\h^{-1}}\ast {a'}\op_{t'}}\otimes_{\Delta^{\C}}m.\\
\end{array}
\]
\textit{\textbf{Step 2.}} Let $\tilde{M}$ be a $\oG$-graded $(A,A')$-bimodule over $\C$, and denote its 1-component by $M:=\tilde{M}_1$. We  prove that $M$ is a $\Delta^{\C}$-module. Let $a_{\g}\in A_{\g}$, ${a'}\op_{\g}\in {A'}\op_{\g}$ and $m\in M$. We define
\[\para{a_{\g}\otimes_{\C} {a'}\op_{\g}}m:=a_{\g}ma'_{\g^{-1}}.\]
To see that this is well-defined, consider $\g\in\oG$ and $a_{\g}\in A_{\g}$, ${a'}\op_{\g}\in {A'}\op_{\g}$ and some $\x,\y\in\oG$ such that $\x\y=\g$ and $a_{\x}\in A_{\x}$, $c_{\y}\in\C_{\y}$ such that $a_{\g}=a_{\x}c_{\y}$ therefore $a_{\g}\otimes_{\C} {a'}\op_{\g} = a_{\x}\otimes_{\C} c_{\y}{a'}\op_{\g}$. We have:
\[
\begin{array}{rcl}
\para{a_{\x}\otimes_{\C} c_{\y}{a'}\op_{\g}}m & = & a_{\x}mc_{\y}{a'}\op_{\g};
\end{array}
\]
but $m\in M$ and $\tilde{M}$ is a $\oG$-graded $(A,A')$-bimodule over $\C$, hence
\[
\begin{array}{rcl}
\para{a_{\x}\otimes_{\C} c_{\y}{a'}\op_{\g}}m & = & a_{\x}c_{\y}m{a'}\op_{\g}\\
 & = & a_{\g}m{a'}\op_{\g}\\
 & = & a_{\g}ma'_{\g^{-1}}.
\end{array}
\]
We verify that $M$ is a $\Delta^{\C}$-module. Consider $\g,\h\in\oG$ and $a_{\g}\in A_{\g}$, $a_{\h}\in A_{\h}$, ${a'}\op_{\g}\in {A'}\op_{\g}$, ${a'}\op_{\h}\in {A'}\op_{\h}$, $m\in M$. We prove that:
\[
\begin{array}{rcl}
	\para{\para{a_{\g}\otimes_{\C} {a'}\op_{\g}}\para{a_{\h}\otimes_{\C} {a'}\op_{\h}}}m = \para{a_{\g}\otimes_{\C} {a'}\op_{\g}}\para{\para{a_{\h}\otimes_{\C} {a'}\op_{\h}}m}.
\end{array}
\]
We have:
\[
\begin{array}{rcl}
	\para{\para{a_{\g}\otimes_{\C} {a'}\op_{\g}}\para{a_{\h}\otimes_{\C} {a'}\op_{\h}}}m & = & \para{a_{\g}a_{\h}\otimes_{\C} {a'}\op_{\g}\ast {a'}\op_{\h}}m\\
	& = & \para{a_{\g}a_{\h}\otimes_{\C} a'_{\h^{-1}}a'_{\g^{-1}}}m\\
	& = & a_{\g}a_{\h}ma'_{\h^{-1}}a'_{\g^{-1}},
\end{array}\]
and
\[\begin{array}{rcl}
	\para{a_{\g}\otimes_{\C} {a'}\op_{\g}}\para{\para{a_{\h}\otimes_{\C} {a'}\op_{\h}}m} & = & \para{a_{\g}\otimes_{\C} {a'}\op_{\g}}\para{a_{\h}ma'_{\h^{-1}}}\\
	& = & a_{\g}a_{\h}ma'_{\h^{-1}}a'_{\g^{-1}}.
\end{array}\]
Therefore, $M$ is a $\Delta^{\C}$-module.

\noindent\textit{\textbf{Step 3.}} We prove that the  functors given in the statement of the theorem are well-defined. Let $\tilde{M},\tilde{N}$ be two $\oG$-graded $(A,A')$-bimodules over $\C$, with 1-components $M:=\tilde{M}_1$ and $N:=\tilde{N}_1$ respectively. Let $\tilde{f}:\tilde{M}\to\tilde{N}$ be a $\oG$-graded $(A,A')$-bimodule homomorphism. Its restriction to $M$ and corestriction to $N$ gives the map
\[(\tilde{f})_1:=f:M\to N.\]
We prove that $f$ is a $\Delta^{\C}$-module homomorphism. We only check the multiplicative property. Let $a_{\g}\in A_{\g}$, ${a'}\op_{\g}\in {A'}\op_{\g}$ and $m\in M$. We have:
\[\begin{array}{rcl}
	f(\para{a_{\g}\otimes_{\C} {a'}\op_{\g}}m) & = & \tilde{f}(\para{a_{\g}\otimes_{\C} {a'}\op_{\g}}m)\\
	& = & \tilde{f}(a_{\g}ma'_{\g^{-1}}),
\end{array}\]
but because $\tilde{f}$ is a $\oG$-graded $(A,A')$-bimodule homomorphism, we obtain that
\[\begin{array}{rcl}
	f(\para{a_{\g}\otimes_{\C} {a'}\op_{\g}}m) & = & a_{\g}\tilde{f}(m)a'_{\g^{-1}}\\
	& = & \para{a_{\g}\otimes_{\C} {a'}\op_{\g}}\tilde{f}(m)\\
	& = & \para{a_{\g}\otimes_{\C} {a'}\op_{\g}}f(m).
\end{array}\]
Now let $M,N$ be two $\Delta^{\C}$-modules, and $f:M\to N$ a $\Delta^{\C}$-module homomorphism. We have that the
\[\begin{array}{l}
\para{A\otimes_{\C} {A'}\op}\otimes_{\Delta^{\C}}f:\para{A\otimes_{\C} {A'}\op}\otimes_{\Delta^{\C}}M\to\para{A\otimes_{\C} {A'}\op}\otimes_{\Delta^{\C}}N,  \\
\para{a\otimes_{\C} {a'}\op}\otimes_{\Delta^{\C}}m \mapsto \para{a\otimes_{\C} {a'}\op}\otimes_{\Delta^{\C}}f(m),
\end{array}\]
 where $a\in A$, ${a'}\op\in {A'}\op$ and $m\in M$, is clearly well-defined. It is straightforward to show that $\para{A\otimes_{\C} {A'}\op}\otimes_{\Delta^{\C}}f$ is a $\oG$-graded $(A,A')$-bimodule homomorphism.

\noindent\textit{\textbf{Step 4.}} We prove that the categories $\Delta^{\C}$-mod and $A\textrm{-}\mathrm{Gr}/\C{\textrm-}A'$ are equivalent. Let $M$ be a $\Delta^{\C}$-module. By Step 1, we know that $(A\otimes_{\C} {A'}\op)\otimes_{\Delta^{\C}} M$ is a $\oG$-graded $(A,A')$-bimodule over $\C$, therefore, by Step 2, we have that $\para{(A\otimes_{\C} {A'}\op)\otimes_{\Delta^{\C}} M}_1$ is a $\Delta^{\C}$-module. It is clear that $\para{(A\otimes_{\C} {A'}\op)\otimes_{\Delta^{\C}} M}_1$ and $M$ are isomorphic as $\Delta^{\C}$-modules.

Let $\tilde{M}$ be a $\oG$-graded $(A,A')$-bimodule over $\C$. We prove that $(A\otimes_{\C} {A'}\op)\otimes_{\Delta^{\C}}\tilde{M}_1$ and $\tilde{M}$ are isomorphic as $\oG$-graded $(A,A')$-bimodules over $\C$. Indeed, by Step 2, we know that $\tilde{M}_1$ is a $\Delta^{\C}$-module. By Step 1, we know that $(A\otimes_{\C} {A'}\op)\otimes_{\Delta^{\C}}\tilde{M}_1$ is a $\oG$-graded $(A,A')$-bimodule over $\C$. Consider now the map
\[
\Phi:(A\otimes_{\C} {A'}\op)\otimes_{\Delta^{\C}}\tilde{M}_1 \to \tilde{M}, \qquad (a_{\g}\otimes_{\C} {a'}\op_{\h})\otimes_{\Delta^{\C}}\tilde{m}_1 \mapsto a_{\g}\tilde{m}_1a'_{\h^{-1}}
\]
for all $\g,\h\in\oG$ and for all $a_{\g}\in A_{\g}$, ${a'}\op_{\h}\in {A'}\op_{\h}$, $\tilde{m}_1\in \tilde{M}_1$. We prove that $\Phi$ is a $\oG$-graded $(A,A')$-bimodule homomorphism. Indeed, we only check the multiplicative  and the grade preservation properties. Let $\x,\y,\g,\h\in \oG$, $a_{\x}\in A_{\x}$, $a_{\g}\in A_{\g}$, ${a'}\op_{\h}\in {A'}\op_{\h}$, $a'_{\y}\in A'_{\y}$, $\tilde{m}_1\in \tilde{M}_1$. We have:
\[\begin{array}{rcl}
	a_{\x}\Phi\para{(a_{\g}\otimes_{\C} {a'}\op_{\h})\otimes_{\Delta^{\C}}\tilde{m}_1}a'_{\y} & = & a_{\x}a_{\g}\tilde{m}_1a'_{\h^{-1}}a'_{\y},
\end{array}\]
and
\[
\begin{array}{rcl}
	\Phi\para{a_{\x}\para{(a_{\g}\otimes_{\C} {a'}\op_{\h})\otimes_{\Delta^{\C}}\tilde{m}_1}a'_{\y}} & = & \Phi\para{(a_{\x}a_{\g}\otimes_{\C} {a'}\op_{\y^{-1}}\ast {a'}\op_{\h})\otimes_{\Delta^{\C}}\tilde{m}_1}\\
	& = & \Phi\para{(a_{\x}a_{\g}\otimes_{\C} (a'_{\h^{-1}}a'_{\y})\op )\otimes_{\Delta^{\C}}\tilde{m}_1}\\
	& = &  a_{\x}a_{\g}\tilde{m}_1a'_{\h^{-1}}a'_{\y}.
\end{array}
\]
For the grade preservation property, consider $\x,\g,\h\in \oG$ such that $\x=\g\h^{-1}$ and $a_{\g}\in A_{\g}$, ${a'}\op_{\h}\in {A'}\op_{\h}$, $\tilde{m}_1\in \tilde{M}_1$. We have that $(a_{\g}\otimes_{\C} {a'}\op_{\h})\otimes_{\Delta^{\C}}\tilde{m}_1$ is of degree $\x$ in $(A\otimes_{\C} {A'}\op)\otimes_{\Delta^{\C}}\tilde{M}_1$ and $\Phi\para{(a_{\g}\otimes_{\C} {a'}\op_{\h})\otimes_{\Delta^{\C}}\tilde{m}_1}=a_{\g}\tilde{m}_1a'_{\h^{-1}}\in \tilde{M}_{\g\h^{-1}}=\tilde{M}_{\x}$. Henceforth, $\Phi$ is a $\oG$-graded $(A,A')$-bimodule homomorphism over $\C$. In particular, $\Phi$ is a morphism in the category of all $\oG$-graded $A$-modules, denoted by $A$-Gr. Because $A$ is strongly $\oG$-graded, it is well known \cite[Theorem 2.8.]{article:Dade1980} that we have the following equivalences of categories:
\[\xymatrix@C+=4cm{
B\textrm{-}\text{Mod} \ar@<+.5ex>@{->}[r]^{A\otimes_B-} & \ar@<+.5ex>@{->}[l]^{\left(-\right)_1} A\textrm{-}\mathrm{Gr},
}\]
where $B$-{Mod} is the category of all $B$-modules. Therefore, the $1$-component of $\Phi$, $\Phi_1$ is a morphism in the category $B$-\text{mod}. We have that
\[\Phi_1:\para{(A\otimes_{\C} {A'}\op)\otimes_{\Delta^{\C}}\tilde{M}_1}_1 \to \tilde{M}_1,\] but $\para{(A\otimes_{\C} {A'}\op)\otimes_{\Delta^{\C}}\tilde{M}_1}_1=\para{A\otimes_{\C} {A'}\op}_1\otimes_{\Delta^{\C}}\tilde{M}_1=\Delta^{\C} \otimes_{\Delta^{\C}}\tilde{M}_1$, so we have:
\[\Phi_1:\Delta^{\C}\otimes_{\Delta^{\C}}\tilde{M}_1 \to \tilde{M}_1,\qquad \Phi_1\para{(a_{\g}\otimes_{\C} {a'}\op_{\g})\otimes_{\Delta^{\C}}\tilde{m}_1}=a_{\g}\tilde{m}_1a'_{\g^{-1}},\]
for all $\g\in\oG$ and for all $a_{\g}\in A_{\g}$, ${a'}\op_{\g}\in {A'}\op_{\g}$, $\tilde{m}_1\in \tilde{M}_1$. Clearly, $\Phi_1$ is an isomorphism of $B$-{modules}. Therefore, by \cite[Corollary 2.10.]{article:Dade1980}, we have that $\Phi$ is an isomorphism in $A$-Gr, thus $\Phi$ is a bijection. Also, because we proved that $\Phi$ is a $\oG$-graded $(A,A')$-bimodule homomorphism over $\C$, we obtain that $(A\otimes_{\C} {A'}\op) \otimes_{\Delta^{\C}}\tilde{M}_1$ and $\tilde{M}$ are isomorphic as $\oG$-graded $(A,A')$-bimodules over $\C$.
\end{proof}

\begin{prop} \label{prop:equivalent_functors} The functors
\[\para{A\otimes_{\C} {A'}\op}\otimes_{\Delta^{\C}}-,\ A\otimes_B -,\ -\otimes_{B'}A':\Delta^{\C}\Md\to A\textrm{-}\mathrm{Gr}/\C{\textrm-}A'\]
are naturally isomorphic equivalences of categories, and their inverse is taking the $1$-component $(-)_1$.
\end{prop}

\begin{proof} In Theorem \ref{th:equivalent_categories}, we have already proven that $\Delta^{\C}$-mod and $A\textrm{-}\mathrm{Gr}/\C{\textrm-}A'$ are equivalent categories, and that their equivalence is given by the functor $\para{A\otimes_{\C} {A'}\op}\otimes_{\Delta^{\C}}-$, which is well-defined and its inverse is $(-)_1$. We complete the proof by also checking the properties of the functors $A\otimes_B -$ and $-\otimes_{B'}A'$.


\noindent \textit{\textbf{Step 1.}} Let $M$  be a  $\Delta^{\C}$-module. We show that $A\otimes_B M$ is a $\oG$-graded $(A,A')$-bimodule over $\C$. The right  $A'$-module structure is defined as follows: For every $\h\in \oG$ and for every $a\in A$, $m\in M$ and $a'_{\h}\in A'_{\h}$, let
\[(a\otimes_B m)\cdot a'_{\h}:=a\cdot u_{\h}\otimes_B (u^{-1}_{\h}\otimes_{\C} {a'}\op_{\h^{-1}})m.\]
 We begin by verifying that this definition does not depend on the choice of $u_{\h}$. Indeed, let $v_{\h}$ be another invertible homogeneous element of $A_{\h}$, hence there exists an invertible element of $B$, $b$, such that $v_{\h}=u_{\h}b$. We compute:
\begingroup
\allowdisplaybreaks
\[
\begin{array}{rcl}
	a\cdot v_{\h}\otimes_B (v^{-1}_{\h}\otimes_{\C} {a'}\op_{{\h}^{-1}})m & = & a\cdot u_{\h}b\otimes_B ((u_{\h}b)^{-1}\otimes_{\C} {a'}\op_{\h^{-1}})m\\
	& = & a\cdot {u}{_{\h}} b\otimes_B (b^{-1}{u}_{\h}^{-1}\otimes_{\C} {a'}\op_{\h^{-1}})m\\
	& = & a\cdot {u}{_{\h}} \otimes_B (bb^{-1}{u}_{\h}^{-1}\otimes_{\C} {a'}\op_{\h^{-1}})m\\
	& = & a\cdot {u}{_{\h}} \otimes_B ({u}_{\h}^{-1}\otimes_{\C} {a'}\op_{\h^{-1}})m.\\
\end{array}
\]
\endgroup
We verify  the definition is $B$-balanced. Indeed, for every $b\in B$, we have that:
\[
\begin{array}{rcl}
(ab\otimes _B m)a'_{\h} & = & ab\cdot u_{\h}\otimes_B (u^{-1}_{\h}\otimes_{\C} {a'}\op_{\h^{-1}})m
\end{array}
\]
and
\[
\begin{array}{rcl}
(a\otimes _B bm)a'_{\h} & = & a u_{\h}\otimes_B (u^{-1}_{\h}\otimes_{\C} {a'}\op_{\h^{-1}})bm\\
 & = & a u_{\h}\otimes_B (u^{-1}_{\h}\otimes_{\C} {a'}\op_{\h^{-1}})(b\otimes_{\C} 1)m\\
 & = & a u_{\h}\otimes_B (u^{-1}_{\h}b\otimes_{\C} {a'}\op_{\h^{-1}})m\\
 & = & a u_{\h}\otimes_B (u^{-1}_{\h}bu_{\h}u^{-1}_{\h}\otimes_{\C} {a'}\op_{\h^{-1}})m\\
 & = & a u_{\h}u^{-1}_{\h}bu_{\h}\otimes_B (u^{-1}_{\h}\otimes_{\C} {a'}\op_{\h^{-1}})m\\
 & = & a bu_{\h}\otimes_B (u^{-1}_{\h}\otimes_{\C} {a'}\op_{\h^{-1}})m.\\
\end{array}
\]
 We verify that $A\otimes_B M$ is a right  $A'$-module. We only check that for every $\g,\h\in\oG$ and for every $a\in A$, $a'_{\g}\in {A'_{\g}},a'_{\h}\in A'_{\h}$, we have:
\[(a\otimes_B m)(a'_{\g}\cdot a'_{\h})=((a\otimes_B m)a'_{\g})a'_{\h}.\]
Indeed, for the left part of the equality, we have for $a'_{\g}\in {A'_{\g}}$ and $a'_{\h}\in A'_{\h}$ that $a'_{\g}a'_{\h}\in {A'_{\g}}A'_{\h}\subseteq A'_{\g\h}$, thus
\[
\begin{array}{rcl}
	(a\otimes_B m)(a'_{\g}\cdot a'_{\h}) & = & au_{\g\h} \otimes _B (u_{\g\h}^{-1}\otimes_{\C} (a'_{\g}a'_{\h})\op)m,
\end{array}
\]
but $u_{\g\h}$, $u_{\g}$ and $u_{\h}$ being invertible homogeneous elements, there exists an invertible element $u\in B$ such that $u_{\g\h}=u_{\g}u_{\h}u$, henceforth
\[
\begin{array}{rcl}
	(a\otimes_B m)(a'_{\g}\cdot a'_{\h}) & = & au_{\g}u_{\h}u \otimes _B (\para{u_{\g}u_{\h}u}^{-1}\otimes_{\C} (a'_{\g}a'_{\h})\op)m\\
	& = & au_{\g}u_{\h}u \otimes _B (u^{-1}\para{u_{\g}u_{\h}}^{-1}\otimes_{\C} (a'_{\g}a'_{\h})\op)m\\
	& = & au_{\g}u_{\h} \otimes _B u(u^{-1}\para{u_{\g}u_{\h}}^{-1}\otimes_{\C} (a'_{\g}a'_{\h})\op)m\\
	& = & au_{\g}u_{\h} \otimes _B (uu^{-1}\para{u_{\g}u_{\h}}^{-1}\otimes_{\C} (a'_{\g}a'_{\h})\op)m\\
	& = & au_{\g}u_{\h} \otimes _B (\para{u_{\g}u_{\h}}^{-1}\otimes_{\C} (a'_{\g}a'_{\h})\op)m.
\end{array}
\]
For the right part of the equality we have
\[
\begin{array}{rcl}
	((a\otimes_B m)a'_{\g})a'_{\h} & = & ( a u_{\g}\otimes_B (u^{-1}_{\g}\otimes_{\C} {a'}\op_{\g^{-1}})m )a'_{\h}\\
	& = &  a u_{\g}u_{\h}\otimes_B (u^{-1}_{\h}\otimes_{\C} {a'}\op_{\h^{-1}})(u^{-1}_{\g}\otimes_{\C} {a'}\op_{\g^{-1}})m \\
	& = &  a u_{\g}u_{\h}\otimes_B (u^{-1}_{\h}u^{-1}_{\g}\otimes_{\C} {a'}\op_{\h^{-1}}\ast {a'}\op_{\g^{-1}})m \\
	& = & au_{\g}u_{\h} \otimes _B (\para{u_{\g}u_{\h}}^{-1}\otimes_{\C} (a'_{\g}a'_{\h})\op)m.
\end{array}
\]
To verify that $A\otimes_B M$ is an $(A,A')$-bimodule, it remains to prove that for every $\h\in\oG$ and for every $\alpha$, $a\in A$, $m\in M$ and $a'_{\h}\in A'_{\h}$, we have
\[(\alpha(a\otimes_B m))a'_{\h}=\alpha((a\otimes_B m)a'_{\h}).\]
Indeed, we have:
\[
\begin{array}{rcl}
(\alpha(a\otimes_B m))a'_{\h} & = & (\alpha a\otimes_B m)a'_{\h} \\
& = & \alpha a u_{\h}\otimes_B (u^{-1}_{\h}\otimes_{\C} {a'}\op_{\h^{-1}})m, \\
\end{array}
\]
and
\begingroup
\allowdisplaybreaks
\[
\begin{array}{rcl}
\alpha((a\otimes_B m)a'_{\h}) & = & \alpha(a u_{\h}\otimes_B (u^{-1}_{\h}\otimes_{\C} {a'}\op_{\h^{-1}})m)\\
& = & \alpha a u_{\h}\otimes_B (u^{-1}_{\h}\otimes_{\C} {a'}\op_{\h^{-1}})m.\\
\end{array}
\]
\endgroup
We define the $x$-component $A\otimes_B M$, for every $x\in\oG$, by
\[(A\otimes_B M)_x:=A_x\otimes _ B M.\]
We  verify that
\[A_{\g}(A\otimes _B M)_x A'_{\h}\subseteq (A\otimes_B M)_{\g x\h},\,\forall \g,x,\h\in\oG.\]
Indeed, for every $a_{\g}\in A_{\g},a_x\in A_x, m\in M, a'_{\h}\in A'_{\h}$, we have:
\[
\begin{array}{rcl}
a_{\g}(a_x\otimes_B m)a'_{\h} & = & (a_{\g}a_x\otimes_B m)a'_{\h}\\
& = & a_{\g}a_x u_{\h}\otimes_B (u^{-1}_{\h}\otimes_{\C} {a'}\op_{\h^{-1}})m,\\
\end{array}
\]
where $a_{\g}a_x u_{\h}\in A_{\g x\h}$ and $(u^{-1}_{\h}\otimes_{\C} {a'}\op_{\h^{-1}})m\in M$, hence  $a_{\g}a_x u_{\h}\otimes_B (u^{-1}_{\h}\otimes_{\C} {a'}\op_{\h^{-1}})m\in (A\otimes_B M)_{\g x\h}$.

We show that $A\otimes_B M$ is an $(A,A')$-bimodule over $\C$. Indeed, for every $\g,\h\in \oG$ and for every $a_{\g}\in A_{\g}$, $m\in M$, $c_{\h}\in\C_h$ we have
\[
\begin{array}{rcl}
(a_{\g}\otimes_B m)c_{\h} & = & a_{\g}u_{\h}\otimes_B(u^{-1}_{\h}\otimes_{\C} c\op_{\h^{-1}}) m\\
& = & a_{\g}u_{\h}\otimes_B(u^{-1}_{\h}c_{\h}\otimes_{\C} 1) m\\
& = & a_{\g}u_{\h}u^{-1}_{\h}c_{\h}\otimes_B(1\otimes_{\C} 1) m\\
& = & a_{\g}c_{\h}\otimes_B m,\\
\end{array}
\]
and
\[
\begin{array}{rcl}
	\tensor[^{\g}]{c}{_{\h}}(a_{\g}\otimes _B m) & = & \tensor[^{\g}]{c}{_{\h}}a_{\g}\otimes _B m\\
	& = & u_{\g}c_{\h}u^{-1}_{\g}a_{\g}\otimes _B m;\\
\end{array}
\]
but $u^{-1}_{\g}a_{\g}\in B$ and $c_{\h}\in C_A(B)$, thus
\[
\begin{array}{rcl}
	\tensor[^{\g}]{c}{_{\h}}(a_{\g}\otimes _B m) & = & u_{\g}u^{-1}_{\g}a_{\g} c_{\h}\otimes _B m\\
	 & = & a_{\g} c_{\h}\otimes _B m.\\
\end{array}
\]


\noindent \textit{\textbf{Step 2.}} Let $M$  be $\Delta^{\C}$-module. We  prove that $M\otimes_{B'}A'$ is a $\oG$-graded $(A,A')$-bimodule over $\C$. The left  $A$-module structure is defined as follows:
For every $\g\in \oG$ and for every $a'\in A'$, $m\in M$ and $a_{\g}\in A_{\g}$,
\[a_{\g}(m\otimes_{B'}a'):=(a_{\g}\otimes_{\C} (u'^{-1}_{\g})\op)m\otimes_{B'}u'_{\g}a'.\]
We begin by verifying that this definition does not depend on the choice of $u'_{\g}$. Let $v'_{\g}$ be another invertible homogeneous element of $A'_{\g}$, hence there exists an invertible element of $B'$, ${b'}$, such that $v'_{\g}={b'}u'_{\g}$. We compute:
\begingroup
\allowdisplaybreaks
\begin{align*}
	(a_{\g}\otimes_{\C} (v'^{-1}_{\g})\op))m\otimes_{B'}v'_{\g}a'& = (a_{\g}\otimes_{\C} (({b'}u'_{\g})^{-1})\op)m\otimes_{B'}({b'}u'_{\g})a'\\
	 &=  (a_{\g}\otimes_{\C} (u'^{-1}_{\g}{b'}^{-1})\op)m\otimes_{B'}{b'}u'_{\g}a'\\
	 &=  (a_{\g}\otimes_{\C} ({b'}^{-1})\op\ast (u'^{-1}_{\g})\op)m\otimes_{B'}{b'}u'_{\g}a'\\
	 &=  (a_{\g}\otimes_{\C} {b'}\op \ast ({b'}^{-1})\op\ast (u'^{-1}_{\g})\op)m \otimes_{B'}u'_{\g}a'\\
	 &=  (a_{\g}\otimes_{\C} ({b'}^{-1} {b'})\op\ast (u'^{-1}_{\g})\op)m\otimes_{B'}u'_{\g}a'\\
	 &=  (a_{\g}\otimes_{\C} (u'^{-1}_{\g})\op)m\otimes_{B'}u'_{\g}a'.
\end{align*}
\endgroup
We verify that the definition is $B'$-balanced. Indeed, for every ${b'}\in B'$, we have:
\[
\begin{array}{rcl}
a_{\g}(m{b'}\otimes_{B'}a') & = & (a_{\g}\otimes_{\C} (u'^{-1}_{\g})\op)(m{b'})\otimes_{B'}u'_{\g}a'\\
& = & (a_{\g}\otimes_{\C} ({b'}u'^{-1}_{\g})\op)m\otimes_{B'}u'_{\g}a'\\
\end{array}
\]
and
\[
\begin{array}{rcl}
a_{\g}(m\otimes_{B'}{b'}a') & = & (a_{\g}\otimes_{\C} (u'^{-1}_{\g})\op)m\otimes_{B'}u'_{\g}{b'}a'\\
& = & (a_{\g}\otimes_{\C} (u'^{-1}_{\g})\op)m\otimes_{B'}u'_{\g}{b'}u'^{-1}_{\g}u'_{\g}a'\\
& = & ((a_{\g}\otimes_{\C} (u'^{-1}_{\g})\op)m)u'_{\g}{b'}u'^{-1}_{\g}\otimes_{B'}u'_{\g}a'\\
& = & (a_{\g}\otimes_{\C} (u'^{-1}_{\g}u'_{\g}{b'}u'^{-1}_{\g})\op)m\otimes_{B'}u'_{\g}a'\\
& = & (a_{\g}\otimes_{\C} ({b'}u'^{-1}_{\g})\op)m\otimes_{B'}u'_{\g}a'.   \\
\end{array}
\]
We verify that $M\otimes_{B'}A'$ is a left $A$-module. Here we only check that for every $\g,\h\in\oG$ and for every $a'\in A'$, $a_{\g}\in {A_{\g}},a_{\h}\in A_{\h}$, we have:
\[(a_{\g}\cdot a_{\h})(m\otimes_{B'} a') = a_{\g}( a_{\h}(m\otimes_{B'} a')).\]
Indeed, we have:
\[
\begin{array}{rcl}
	(a_{\g}\cdot a_{\h})(m\otimes_{B'} a') & = & (a_{\g}a_{\h}\otimes_{\C} ((u'_{\g}u'_{\h})^{-1})\op)m\otimes_{B'}u'_{\g}u'_{\h}a'\\
	& = & (a_{\g}a_{\h}\otimes_{\C} (u'^{-1}_{\h}u'^{-1}_{\g})\op)m\otimes_{B'}u'_{\g}u'_{\h}a'\\
	& = & (a_{\g}a_{\h}\otimes_{\C} (u'^{-1}_{\g})\op\ast (u'^{-1}_{\h})\op)m\otimes_{B'}u'_{\g}u'_{\h}a'\\
\end{array}
\]
and
\[
\begin{array}{rcl}
	a_{\g}( a_{\h}(m\otimes_{B'} a')) & = & a_{\g}( (a_{\h}\otimes_{\C} (u'^{-1}_{\h})\op)m\otimes_{B'}u'_{\h}a')\\
	& = & (a_{\g}\otimes_{\C} (u'^{-1}_{\g})\op)(a_{\h}\otimes_{\C} (u'^{-1}_{\h})\op)m\otimes_{B'}u'_{\g}u'_{\h}a'\\
	& = & (a_{\g}a_{\h}\otimes_{\C} (u'^{-1}_{\g})\op\ast (u'^{-1}_{\h})\op)m\otimes_{B'}u'_{\g}u'_{\h}a'.\\
\end{array}
\]
 To show that $M\otimes_{B'}A'$ is an $(A,A')$-bimodule, it remains to prove that for every $\g\in\oG$ and for every $\alpha',a'\in A'$, $m\in M$ and $a_{\g}\in A_{\g}$, we have:
\[(a_{\g}(m\otimes_{B'} a'))\alpha'=a_{\g}((m\otimes_{B'} a')\alpha').\]
Indeed,
\[\begin{array}{rcl}
\para{a_{\g}(m\otimes_{B'} a')}\alpha' & = & \para{(a_{\g}\otimes_{\C} (u'^{-1}_{\g})\op)m\otimes_{B'}u'_{\g}a'}\alpha'\\
& = & (a_{\g}\otimes_{\C} (u'^{-1}_{\g})\op)m\otimes_{B'}u'_{\g}a'\alpha'
\end{array}\]
and
\[
\begin{array}{rcl}
a_{\g}((m\otimes_{B'} a')\alpha') & = & a_{\g}(m\otimes_{B'} a'\alpha')\\
& = & (a_{\g}\otimes_{\C} (u'^{-1}_{\g})\op)m\otimes_{B'}u'_{\g}a'\alpha'.
\end{array}
\]
The $x$-component of $M\otimes_{B'}A'$, for every $x\in\oG$, is, by definition,
\[(M\otimes_{B'}A')_x:=M\otimes_{B'}A'_x.\]
We  verify that
\[A_{\g}(M\otimes_{B'}A')_x A'_{\h}\subseteq (M\otimes_{B'}A')_{\g x\h},\,\forall \g,x,\h\in\oG.\]
Indeed, for every $a_{\g}\in A_{\g}$, $a'_x\in A'_x$, $m\in M, a'_{\h}\in A'_{\h}$, we have:
\[
\begin{array}{rcl}
a_{\g}(m\otimes_{B'} a'_x)a'_{\h} & = & a_{\g}(m\otimes_{B'} a'_xa'_{\h})\\
& = & (a_{\g}\otimes_{\C} (u'^{-1}_{\g})\op)m\otimes_{B'}u'_{\g} a'_xa'_{\h}),\\
\end{array}
\]
where $u'_{\g} a'_xa'_{\h}\in A'_{\g x\h}$ and $(a_{\g}\otimes_{\C} (u'^{-1}_{\g})\op)m\in M$, hence  $(a_{\g}\otimes_{\C} (u'^{-1}_{\g})\op)m\otimes_{B'}u'_{\g} a'_xa'_{\h})\in (M\otimes_{B'}A')_{\g x\h}$.

We  verify that $M\otimes_{B'}A'$ is an $(A,A')$-bimodule over $\C$. For all $\g,\h\in \oG$ and for all $a'_{\g}\in A'_{\g}$, $m\in M$, $c_{\h}\in\C_h$ we have
\[\begin{array}{rcl}
( m\otimes_{B'}a'_{\g})\tensor[^{{\g}^{-1}}]{c}{_{\h}} & = &  m\otimes_{B'}a'_{\g}\tensor[^{{\g}^{-1}}]{c}{_{\h}}\\
& = &  m\otimes_{B'}a'_{\g}u'^{-1}_{\g}c_{\h}u'_{\g}\\
\end{array}
\]
and
\[
\begin{array}{rcl}
c_{\h}( m\otimes_{B'}a'_{\g}) & = & (c_{\h}\otimes_{\C} (u'^{-1}_{\h})\op)m\otimes_{B'}u'_{\h}a'_{\g}\\
& = & (1\otimes_{\C} (c_{\h}u'^{-1}_{\h})\op)m\otimes_{B'}u'_{\h}a'_{\g},\\
\end{array}
\]
but $c_{\h}u'^{-1}_{\h}\in B'$, thus
\[
\begin{array}{rcl}
c_{\h}( m\otimes_{B'}a'_{\g}) & = & m\otimes_{B'}c_{\h}u'^{-1}_{\h}u'_{\h}a'_{\g}\\
& = & m\otimes_{B'}c_{\h}a'_{\g};\\
\end{array}
\]
 if we take $a'_{\g}=u'_{\g}{b'}$, for some ${b'}\in B'$, and because $u'^{-1}_{\g}c_{\h}u'_{\g}\in C_{A'}(B')$ we have
\[
\begin{array}{rcl}
c_{\h}( m\otimes_{B'}a'_{\g})& = & m\otimes_{B'}u'_{\g}u'^{-1}_{\g}c_{\h}u'_{\g}{b'}\\
& = & m\otimes_{B'}u'_{\g}{b'}u'^{-1}_{\g}c_{\h}u'_{\g}\\
& = & m\otimes_{B'}a'_{\g}u'^{-1}_{\g}c_{\h}u'_{\g}.\\
\end{array} \]
\noindent \textit{\textbf{Step 3.}} We prove that the functors $\mathcal{F}=A\otimes_B -$ and $\mathcal{F}'=-\otimes_{B'}A'$ have as an inverse the functor $\mathcal{G}=(-)_1$.
If $M\in \Delta^{\C}$-{mod}, then we obviously have that ${A\otimes_B M}_1 \simeq M$ as $\Delta^{\C}$-modules. By Theorem \ref{th:equivalent_categories},  the functor $\mathcal{G}$ has an inverse $\mathcal{G}^{-1}=(A\otimes_{\C} {A'}\op)\otimes_{\Delta^{\C}}-$. Since $\mathcal{G}\circ \mathcal{F}\simeq \mathrm{id}_{\Delta^{\C}-\text{mod}}$, we deduce that $\mathcal{F}\simeq \mathcal{G}^{-1}$. By a similar argument we get $\mathcal{F}'\simeq \mathcal{G}^{-1}$, hence we have three naturally isomorphic equivalences of categories as claimed.
\end{proof}

\begin{prop} \label{lemma:hom} For a $\Delta^{\C}$-module $M$, we denote $\widetilde{M}=(A\otimes_{\C} {A'}\op)\otimes_{\Delta^{\C}} M\simeq A\otimes_B M \simeq M\otimes_{B'}A'$. Let $A''$ be a third $\bar G$-graded crossed product over $\mathcal{C}$.
\begin{enumerate}
	\item[{\rm(1)}] Let $M$ be a $\Delta(A\otimes_{\C} {A'}\op)$-module and let $M'$ be a $\Delta(A'\otimes_{\C} {{A''}}\op)$-module. Then $M\otimes_{B'}M'$ is a $\Delta(A\otimes_{\C} {{A''}}\op)$-module with the multiplication operation defined by
	\[(a_{\g}\otimes_{\C} {a''}\op_{\g})(m\otimes_{B'} m'):=(a_{\g}\otimes_{\C} (u'^{-1}_{\g})\op) m\otimes_{B'} (u'_{\g} \otimes_{\C} {a''}\op_{\g})m',\]
	for all $\g\in\oG$ and for all $a_{\g}\in A_{\g},\,{a''}\op_{\g}\in {A''}\op_{\g},\,m\in M,\, m'\in M'$.
	Moreover, we have the isomorphism
	\[\widetilde{M\otimes_{B'} M'}\simeq \widetilde{M}\otimes_{A'} \widetilde{M'}\]
	of $\oG$-graded $(A,{A''})$-bimodules over $\C$.
\item[{\rm(2)}] Let $M$ be a $\Delta(A'\otimes_{\C} A\op)$-module and let $M'$ be a $\Delta(A'\otimes_{\C} {A''}\op)$-module. Then $\Hom{B'}{M}{M'}$ is a $\Delta(A\otimes_{\C} {A''}\op)$-module with the following operation:
	\[((a_{\g}\otimes_{\C} {a''}\op_{\g})f)(m):=(u'_{\g}\otimes_{\C} {a''}\op_{\g})f((u'^{-1}_{\g}\otimes_{\C} a_{\g^{-1}}\op)m),\]
		for all $\g\in\oG$ and for all $a_{\g}\in A_{\g},\,{a''}\op_{\g}\in {A''}\op_{\g},\,m\in M$, $f\in \Hom{B'}{M}{M'}$.
		Moreover, we have the isomorphism
	\[\widetilde{\Hom{B'}{M}{M'}}\simeq \Hom{A'}{\widetilde{M}}{\widetilde{M'}}\]
	of $\oG$-graded $(A,{A''})$-bimodules over $\C$.
\end{enumerate}
\end{prop}

\begin{proof} (1) We prove that $M\otimes_{B'}M'$ is a $\Delta(A\otimes_{\C} {A''}\op)$-module with the given multiplication operation. The  definition does not depend on the choice of the homogeneous invertible elements $u'_{\g}\in A'_{\g}$. Indeed, let $v'_{\g}$ be another invertible homogeneous element of $A'_{\g}$, hence there exists an invertible element ${b'}$ of $B'$,  such that $v'_{\g}={b'}u'_{\g}$. We compute:
\begin{align*}
(a_{\g}&\otimes_{\C} (v'^{-1}_{\g})\op) m\otimes_{B'} (v'_{\g} \otimes_{\C} {a''}\op_{\g})m'\\
  = & (a_{\g}\otimes_{\C} (({b'}u'_{\g})^{-1})\op) m\otimes_{B'} (({b'}u'_{\g}) \otimes_{\C} {a''}\op_{\g})m'\\
  = & (a_{\g}\otimes_{\C} (u'^{-1}_{\g}{b'}^{-1})\op) m\otimes_{B'} (({b'}u'_{\g}) \otimes_{\C} {a''}\op_{\g})m'\\
  = & (a_{\g}\otimes_{\C} {b'}\op\ast(u'^{-1}_{\g}{b'}^{-1})\op) m\otimes_{B'} (u'_{\g} \otimes_{\C} {a''}\op_{\g})m'\\
  = & (a_{\g}\otimes_{\C} (u'^{-1}_{\g}{b'}^{-1}{b'})\op) m\otimes_{B'} (u'_{\g} \otimes_{\C} {a''}\op_{\g})m'\\
  = & (a_{\g}\otimes_{\C} (u'^{-1}_{\g})\op) m\otimes_{B'} (u'_{\g} \otimes_{\C} {a''}\op_{\g})m'.
\end{align*}
We verify that the multiplication definition is $B'$-balanced. Indeed, for every ${b'}\in B'$, we have that:
\begin{align*}
	(a_{\g} & \otimes_{\C} {a''}\op_{\g})(m{b'}\otimes_{B'} m') \\
	= & (a_{\g}\otimes_{\C} (u'^{-1}_{\g})\op) m{b'}\otimes_{B'} (u'_{\g} \otimes_{\C} {a''}\op_{\g})m'\\
	= & (a_{\g}\otimes_{\C} (u'^{-1}_{\g})\op)(1\otimes_{\C} {b'}\op) m\otimes_{B'} (u'_{\g} \otimes_{\C} {a''}\op_{\g})m'\\
	= & (a_{\g}\otimes_{\C} (u'^{-1}_{\g})\op\ast {b'}\op) m\otimes_{B'} (u'_{\g} \otimes_{\C} {a''}\op_{\g})m',
\end{align*}
and
\begingroup
\allowdisplaybreaks
\begin{align*}
	(a_{\g} & \otimes_{\C} {a''}\op_{\g})(m\otimes_{B'} {b'}m') \\
	= & (a_{\g}\otimes_{\C} (u'^{-1}_{\g})\op) m\otimes_{B'} (u'_{\g} \otimes_{\C} {a''}\op_{\g}){b'}m'\\
	= & (a_{\g}\otimes_{\C} (u'^{-1}_{\g})\op) m\otimes_{B'} (u'_{\g} \otimes_{\C} {a''}\op_{\g})({b'}\otimes_{\C} 1)m'\\
	= & (a_{\g}\otimes_{\C} (u'^{-1}_{\g})\op) m\otimes_{B'} (u'_{\g}{b'} \otimes_{\C} {a''}\op_{\g})m'\\
	= & (a_{\g}\otimes_{\C} (u'^{-1}_{\g})\op) m\otimes_{B'} (u'_{\g}{b'}u'^{-1}_{\g}u'_{\g} \otimes_{\C} {a''}\op_{\g})m'\\
	= & (a_{\g}\otimes_{\C} (u'^{-1}_{\g})\op) m\otimes_{B'} (\tensor*[^{\g}]{{{b'}}}{}u'_{\g} \otimes_{\C} {a''}\op_{\g})m';
\end{align*}
\endgroup
since $\tensor*[^{\g}]{{{b'}}}{}\in B'$, we have
\begingroup
\allowdisplaybreaks
\begin{align*}
	(a_{\g} & \otimes_{\C} {a''}\op_{\g})(m\otimes_{B'} {b'}m') \\
	= & \tensor*[^{\g}]{{{b'}}}{}(a_{\g}\otimes_{\C} (u'^{-1}_{\g})\op) m\otimes_{B'} (u'_{\g} \otimes_{\C} {a''}\op_{\g})m'\\
	= & (1\otimes_{\C} (\tensor*[^{\g}]{{{b'}}}{})\op)(a_{\g}\otimes_{\C} (u'^{-1}_{\g})\op) m\otimes_{B'} (u'_{\g} \otimes_{\C} {a''}\op_{\g})m'\\
	= & (a_{\g}\otimes_{\C} (\tensor*[^{\g}]{{{b'}}}{})\op\ast (u'^{-1}_{\g})\op) m\otimes_{B'} (u'_{\g} \otimes_{\C} {a''}\op_{\g})m'\\
	= & (a_{\g}\otimes_{\C} (u'^{-1}_{\g}\tensor*[^{\g}]{{{b'}}}{})\op) m\otimes_{B'} (u'_{\g} \otimes_{\C} {a''}\op_{\g})m'\\
	= & (a_{\g}\otimes_{\C} (u'^{-1}_{\g}u'_{\g}{b'}u'^{-1}_{\g})\op) m\otimes_{B'} (u'_{\g} \otimes_{\C} {a''}\op_{\g})m'\\
	= & (a_{\g}\otimes_{\C} ({b'}u'^{-1}_{\g})\op) m\otimes_{B'} (u'_{\g} \otimes_{\C} {a''}\op_{\g})m'\\
	= & (a_{\g}\otimes_{\C} (u'^{-1}_{\g})\op\ast {b'}\op) m\otimes_{B'} (u'_{\g} \otimes_{\C} {a''}\op_{\g})m'.
\end{align*}
\endgroup
 We verify that $M\otimes_{B'}M'$ is a $\Delta(A\otimes_{\C} {A''}\op)$-module. Let $\g,\,\h\in\oG$ and $a_{\g}\in A_{\g}$, $a_{\h}\in A_{\h}$, ${a''}\op_{\g}\in {A''}\op_{\g},\,{a''}\op_{\h}\in {A''}\op_{\h},\,m\in M,\, m'\in M'$. We only check that
\[((a_{\g}\otimes_{\C} {a''}\op_{\g})(a_{\h}\otimes_{\C} {a''}\op_{\h}))(m\otimes_{B'} m')=(a_{\g}\otimes_{\C} {a''}\op_{\g})((a_{\h}\otimes_{\C} {a''}\op_{\h})(m\otimes_{B'} m')).\]
Indeed,
\begin{align*}
(a_{\g}&\otimes_{\C} {a''}\op_{\g})((a_{\h}\otimes_{\C} {a''}\op_{\h})(m\otimes_{B'} m'))\\
 = & (a_{\g}\otimes_{\C} {a''}\op_{\g})\para{(a_{\h}\otimes_{\C} (u'^{-1}_{\h})\op) m\otimes_{B'} (u'_{\h} \otimes_{\C} {a''}\op_{\h})m'}\\
  = & (a_{\g}\otimes_{\C} (u'^{-1}_{\g})\op)(a_{\h}\otimes_{\C} (u'^{-1}_{\h})\op) m\otimes_{B'} (u'_{\g} \otimes_{\C} {a''}\op_{\g})(u'_{\h} \otimes_{\C} {a''}\op_{\h})m'\\
  = & (a_{\g}a_{\h}\otimes_{\C} (u'^{-1}_{\g})\op\ast (u'^{-1}_{\h})\op) m\otimes_{B'}(u'_{\g}u'_{\h} \otimes_{\C} {a''}\op_{\g}\ast {a''}\op_{\h})m',
\end{align*}
and
\begingroup
\allowdisplaybreaks
\begin{align*}
((a_{\g}&\otimes_{\C} {a''}\op_{\g})(a_{\h}\otimes_{\C} {a''}\op_{\h}))(m\otimes_{B'} m')\\
  = & (a_{\g}a_{\h}\otimes_{\C} {a''}\op_{\g}\ast {a''}\op_{\h})(m\otimes_{B'} m')\\
  = & (a_{\g}a_{\h}\otimes_{\C} (u'^{-1}_{\g\h})\op) m\otimes_{B'} (u'_{\g\h} \otimes_{\C} {a''}\op_{\g}\ast {a''}\op_{\h} )m';
\end{align*}
but $u'_{\g\h}$, $u'_{\g}$ and $u'_{\h}$ being invertible homogeneous elements, there exists  an invertible element $u'\in B'$ such that $u'_{\g\h}=u'u'_{\g}u'_{\h}$, henceforth
\begin{align*}
((a_{\g}&\otimes_{\C} {a''}\op_{\g})(a_{\h}\otimes_{\C} {a''}\op_{\h}))(m\otimes_{B'} m')\\
  = & (a_{\g}a_{\h}\otimes_{\C} ((u'u'_{\g}u'_{\h})^{-1})\op) m\otimes_{B'} (u'u'_{\g}u'_{\h} \otimes_{\C} {a''}\op_{\g}\ast {a''}\op_{\h} )m'\\
  = & (a_{\g}a_{\h}\otimes_{\C} (u'^{-1}_{\h}u'^{-1}_{\g}u'^{-1})\op) m\otimes_{B'} (u'u'_{\g}u'_{\h} \otimes_{\C} {a''}\op_{\g}\ast {a''}\op_{\h} )m'\\
  = & (a_{\g}a_{\h}\otimes_{\C} (u'^{-1})\op \ast (u'^{-1}_{\g})\op\ast(u'^{-1}_{\h})\op) m\otimes_{B'} (u'u'_{\g}u'_{\h} \otimes_{\C} {a''}\op_{\g}\ast {a''}\op_{\h} )m' \\
  = & (1\otimes_{\C} (u'^{-1})\op)(a_{\g}a_{\h}\otimes_{\C} (u'^{-1}_{\g})\op\ast(u'^{-1}_{\h})\op) m\otimes_{B'} (u' u'_{\g}u'_{\h}\otimes_{\C} {a''}\op_{\g}\ast {a''}\op_{\h} )m'\\
  = & (a_{\g}a_{\h}\otimes_{\C} (u'^{-1}_{\g})\op\ast(u'^{-1}_{\h})\op) m\otimes_{B'} (u'^{-1}\otimes_{\C} 1)(u'u'_{\g}u'_{\h} \otimes_{\C} {a''}\op_{\g}\ast {a''}\op_{\h} )m'\\
  = & (a_{\g}a_{\h}\otimes_{\C} (u'^{-1}_{\g})\op\ast(u'^{-1}_{\h})\op) m\otimes_{B'} (u'^{-1}u'u'_{\g}u'_{\h} \otimes_{\C} {a''}\op_{\g}\ast {a''}\op_{\h} )m'\\
  = & (a_{\g}a_{\h}\otimes_{\C} (u'^{-1}_{\g})\op\ast(u'^{-1}_{\h})\op) m\otimes_{B'} (u'_{\g}u'_{\h} \otimes_{\C} {a''}\op_{\g}\ast {a''}\op_{\h} )m'.
\end{align*}
\endgroup
We prove that $\widetilde{M}\otimes_{A'} \widetilde{M'}$ is a $\oG$-graded $(A,{A''})$-bimodule over $\C$. By definition, for $t,t'\in\oG$, $\widetilde{M}_t\otimes_{A'} \widetilde{M'}_{t'}$ is of degree $tt'\in\oG$. It is easy to check that this definition is  correct, and that that $\widetilde{M}\otimes_{A'} \widetilde{M'}$ is a $\oG$-graded $(A,{A''})$-bimodule. Consider, $c\in \C$, $x,t,t'\in \oG$ such that $x=tt'$ and $\widetilde{m}_t\in \widetilde{M}_t,\,\widetilde{m'}_{t'}\in \widetilde{M'}_{t'}$. We have:
\begingroup
\allowdisplaybreaks
\[
\begin{array}{rcl}
(\widetilde{m}_t\otimes_{A'}\widetilde{m'}_{t'} )c & = & \widetilde{m}_t\otimes_{A'}\widetilde{m'}_{t'} c\\
 & = & \widetilde{m}_t\otimes_{A'}\tensor*[^{t'}]{c}{}\widetilde{m'}_{t'} \\
 & = & \widetilde{m}_t\tensor*[^{t'}]{c}{}\otimes_{A'}\widetilde{m'}_{t'} \\
 & = & \tensor*[^{tt'}]{c}{}\widetilde{m}_t\otimes_{A'}\widetilde{m'}_{t'} \\
 & = & \tensor*[^{x}]{c}{}\widetilde{m}_t\otimes_{A'}\widetilde{m'}_{t'}\\
 & = & \tensor*[^{x}]{c}{}(\widetilde{m}_t\otimes_{A'}\widetilde{m'}_{t'}).
\end{array}
\]
Finally, by using Proposition \ref{prop:equivalent_functors}, we get the isomorphisms
\[
\begin{array}{rcl}
\widetilde{M}\otimes_{A'} \widetilde{M'} &\simeq& A\otimes_B M\otimes_{A'} A'\otimes_{B'}M'\\
&\simeq& A\otimes_B M\otimes_{B'}M'\\
&\simeq& A\otimes_B \para{M\otimes_{B'}M'}\\
&\simeq& \widetilde{M\otimes_{B'} M'}
\end{array}\]
\endgroup
of $\oG$-graded $(A,{A''})$-bimodules over $\C$ .

(2) We begin by proving that $\Hom{B'}{M}{M'}$ is a $\Delta(A\otimes_{\C} {A''}\op)$-module. We verify that the given definition does not depend on the choice of homogeneous invertible elements $u'_{\g}\in A'_{\g}$. Indeed, let $v'_{\g}$ be another invertible homogeneous element of $A'_{\g}$, hence there exists an invertible element of $B'$, ${b'}$, such that $v'_{\g}=u'_{\g}{b'}$. We compute:
\[
\begin{array}{rl}
	(v'_{\g}\otimes_{\C} {a''}\op_{\g})f((v'^{-1}_{\g}\otimes_{\C} a_{\g^{-1}}\op)m)
	&= (u'_{\g}{b'}\otimes_{\C} {a''}\op_{\g})f\left(((u'_{\g}{b'})^{-1}\otimes_{\C} a_{\g^{-1}}\op)m\right)\\
	&= (u'_{\g}{b'}\otimes_{\C} {a''}\op_{\g})f\left(({b'}^{-1}u'^{-1}_{\g}\otimes_{\C} a_{\g^{-1}}\op)m\right)\\
	&= (u'_{\g}{b'}\otimes_{\C} {a''}\op_{\g})f\left({b'}^{-1}(u'^{-1}_{\g}\otimes_{\C} a_{\g^{-1}}\op)m\right)\\
	&= (u'_{\g}{b'}\otimes_{\C} {a''}\op_{\g}){b'}^{-1}f\left((u'^{-1}_{\g}\otimes_{\C} a_{\g^{-1}}\op)m\right)\\
	&= (u'_{\g}{b'}{b'}^{-1}\otimes_{\C} {a''}\op_{\g})f\left((u'^{-1}_{\g}\otimes_{\C} a_{\g^{-1}}\op)m\right)\\
	&= (u'_{\g}\otimes_{\C} {a''}\op_{\g})f\left((u'^{-1}_{\g}\otimes_{\C} a_{\g^{-1}}\op)m\right).
\end{array}
\]
We verify that $(a_{\g}\otimes_{\C} {a''}\op_{\g})f$ is  $B'$-linear. Indeed, for ${b'}\in B'$, we compute:
\begingroup
\allowdisplaybreaks
\[
\begin{array}{rcl}
	((a_{\g}\otimes_{\C} {a''}\op_{\g})f)({b'}m) & = & (u'_{\g}\otimes_{\C} {a''}\op_{\g})f((u'^{-1}_{\g}\otimes_{\C} a_{\g^{-1}}\op){b'}m)\\
	 & = & (u'_{\g}\otimes_{\C} {a''}\op_{\g})f((u'^{-1}_{\g}{b'}\otimes_{\C} a_{\g^{-1}}\op)m)\\
	 & = & (u'_{\g}\otimes_{\C} {a''}\op_{\g})f((u'^{-1}_{\g}{b'}u'_{\g}u'^{-1}_{\g}\otimes_{\C} a_{\g^{-1}}\op)m)\\
	 & = & (u'_{\g}\otimes_{\C} {a''}\op_{\g})f(u'^{-1}_{\g}{b'}u'_{\g}(u'^{-1}_{\g}\otimes_{\C} a_{\g^{-1}}\op)m);
\end{array}
\]
\endgroup
but $u'^{-1}_{\g}{b'}u'_{\g}\in B'$ and $f\in \Hom{B'}{M}{M'}$, thus
\begingroup
\allowdisplaybreaks
\[
\begin{array}{rcl}
	((a_{\g}\otimes_{\C} {a''}\op_{\g})f)({b'}m) 
	& = & (u'_{\g}\otimes_{\C} {a''}\op_{\g})u'^{-1}_{\g}{b'}u'_{\g}f((u'^{-1}_{\g}\otimes_{\C} a_{\g^{-1}}\op)m)\\
	& = & (u'_{\g}u'^{-1}_{\g}{b'}u'_{\g}\otimes_{\C} {a''}\op_{\g})f((u'^{-1}_{\g}\otimes_{\C} a_{\g^{-1}}\op)m)\\
	& = & ({b'}u'_{\g}\otimes_{\C} {a''}\op_{\g})f((u'^{-1}_{\g}\otimes_{\C} a_{\g^{-1}}\op)m)\\
	& = & {b'}(u'_{\g}\otimes_{\C} {a''}\op_{\g})f((u'^{-1}_{\g}\otimes_{\C} a_{\g^{-1}}\op)m)\\
	& = & {b'}((a_{\g}\otimes_{\C} {a''}\op_{\g})f)(m).
\end{array}
\]
\endgroup
We verify that $\Hom{B'}{M}{M'}$ is a $\Delta(A\otimes_{\C} {A''}\op)$-module. Let $\g,\,\h\in\oG$, $a_{\g}\in A_{\g},\,a_{\h}\in A_{\h},\,{a''}\op_{\g}\in {A''}\op_{\g},\,{a''}\op_{\h}\in {A''}\op_{\h},\,m\in M$, $f\in \Hom{B'}{M}{M'}$. We only check that
\[(((a_{\g}\otimes_{\C} {a''}\op_{\g})(a_{\h}\otimes_{\C} {a''}\op_{\h}))f)(m)=((a_{\g}\otimes_{\C} {a''}\op_{\g})((a_{\h}\otimes_{\C} {a''}\op_{\h})f))(m).\]
Indeed,
\begin{align*}
((a_{\g}&\otimes_{\C} {a''}\op_{\g})((a_{\h}\otimes_{\C} {a''}\op_{\h})f))(m)\\
 = & (u'_{\g}\otimes_{\C} {a''}\op_{\g}) ((a_{\h}\otimes_{\C} {a''}\op_{\h})f) ((u'^{-1}_{\g}\otimes_{\C} a_{\g^{-1}}\op)m)\\
 = & (u'_{\g}\otimes_{\C} {a''}\op_{\g})(u'_{\h}\otimes_{\C} {a''}\op_{\h}) f ((u'^{-1}_{\h}\otimes_{\C} a_{\h^{-1}}\op)(u'^{-1}_{\g}\otimes_{\C} a_{\g^{-1}}\op)m)\\
 = & (u'_{\g}u'_{\h}\otimes_{\C} {a''}\op_{\g}\ast {a''}\op_{\h}) f ((u'^{-1}_{\h}u'^{-1}_{\g}\otimes_{\C} a_{\h^{-1}}\op\ast a_{\g^{-1}}\op)m),
\end{align*}
and
\begin{align*}
(((a_{\g}&\otimes_{\C} {a''}\op_{\g})(a_{\h}\otimes_{\C} {a''}\op_{\h}))f)(m)\\
=& ((a_{\g}a_{\h}\otimes_{\C} {a''}\op_{\g}\ast {a''}\op_{\h})f)(m)\\
=& ((a_{\g}a_{\h}\otimes_{\C} ({a''}_{\h^{-1}}{a''}_{\g^{-1}})\op )f)(m)\\
=& (u'_{\g\h}\otimes_{\C}({a''}_{\h^{-1}}{a''}_{\g^{-1}})\op)f((u'^{-1}_{\g\h}\otimes_{\C} (a_{\g}a_{\h})\op)m);
\end{align*}
but $u'_{\g\h}$, $u'_{\g}$ and $u'_{\h}$ being invertible homogeneous elements, there exists an  invertible element $u'\in B'$ such that $u'_{\g\h}=u'_{\g}u'_{\h}u'$, henceforth
\begingroup
\allowdisplaybreaks
\begin{align*}
(((a_{\g}&\otimes_{\C} {a''}\op_{\g})(a_{\h}\otimes_{\C} {a''}\op_{\h}))f)(m)\\
=& (u'_{\g}u'_{\h}u'\otimes_{\C}({a''}_{\h^{-1}}{a''}_{\g^{-1}})\op)f(((u'_{\g}u'_{\h}u')^{-1}\otimes_{\C} (a_{\g}a_{\h})\op)m)\\
=& (u'_{\g}u'_{\h}\otimes_{\C}({a''}_{\h^{-1}}{a''}_{\g^{-1}})\op)u'f(((u'_{\g}u'_{\h}u')^{-1}\otimes_{\C} (a_{\g}a_{\h})\op)m)\\
=& (u'_{\g}u'_{\h}\otimes_{\C}({a''}_{\h^{-1}}{a''}_{\g^{-1}})\op)f(u'((u'_{\g}u'_{\h}u')^{-1}\otimes_{\C} (a_{\g}a_{\h})\op)m)\\
=& (u'_{\g}u'_{\h}\otimes_{\C}({a''}_{\h^{-1}}{a''}_{\g^{-1}})\op)f((u'(u'_{\g}u'_{\h}u')^{-1}\otimes_{\C} (a_{\g}a_{\h})\op)m)\\
=& (u'_{\g}u'_{\h}\otimes_{\C}({a''}_{\h^{-1}}{a''}_{\g^{-1}})\op)f((u'u'^{-1}u'^{-1}_{\h}u'^{-1}_{\g}\otimes_{\C} (a_{\g}a_{\h})\op)m)\\
=& (u'_{\g}u'_{\h}\otimes_{\C}({a''}_{\h^{-1}}{a''}_{\g^{-1}})\op)f((u'^{-1}_{\h}u'^{-1}_{\g}\otimes_{\C} (a_{\g}a_{\h})\op)m)\\
 = & (u'_{\g}u'_{\h}\otimes_{\C} {a''}\op_{\g}\ast {a''}\op_{\h}) f ((u'^{-1}_{\h}u'^{-1}_{\g}\otimes_{\C} a_{\h^{-1}}\op\ast a_{\g^{-1}}\op)m).
\end{align*}
\endgroup
We prove that $\Hom{A'}{\widetilde{M}}{\widetilde{M'}}$ is a $\oG$-graded $(A,{A''})$-bimodule over $\C$. By \cite[1.6.4.]{book:Marcus1999}, $\Hom{A'}{\widetilde{M}}{\widetilde{M'}}$ is a $\oG$-graded $(A,{A''})$-bimodule with the $x$-component ($x\in \oG$) defined as follows:
	\[\Hom{A'}{\widetilde{M}}{\widetilde{M'}}_x=\set{f\in \Hom{A'}{\widetilde{M}}{\widetilde{M'}}\mid f(\widetilde{M}_{\g})\subseteq \widetilde{M'}_{\g x}\text{ for all }\g\in\oG}.\]
	It remains to prove that $\Hom{A'}{\widetilde{M}}{\widetilde{M'}}$ is a $\oG$-graded $(A,{A''})$-bimodule over $\C$. Consider $x\in\oG$, $c\in \C$ and $f\in \Hom{A'}{\widetilde{M}}{\widetilde{M'}}_x$.
Let $\g\in\oG$ and $\tilde{m}_{\g}\in\tilde{M}_{\g}$. We show that
	\[(fc)(\tilde{m}_{\g})=(\tensor*[^{x}]{c}{}f)(\tilde{m}_{\g}).\]
Indeed, because $f\in \Hom{A'}{\widetilde{M}}{\widetilde{M'}}_x$ we have that $f\in \Hom{A'}{\widetilde{M}}{\widetilde{M'}}$, hence
\[(fc)(\tilde{m}_{\g})=f(\tilde{m}_{\g})c$ and $(\tensor*[^{x}]{c}{}f)(\tilde{m}_{\g})=f(\tilde{m}_{\g}\tensor*[^{x}]{c}{}).\]
We have that $f(\tilde{m}_{\g})\in \tilde{M}'_{\g x}$. Since $\tilde{M}'$ is a $\oG$-graded $(A',{A''})$-bimodule over $\C$, we have that \[(fc)(\tilde{m}_{\g})=\tensor*[^{\g x}]{c}{}f(\tilde{m}_{\g}).\]
On the right hand side, because  $\tilde{m}_{\g}\in\tilde{M}_{\g}$ and  $\tilde{M}$ is a $\oG$-graded $(A',A)$-bimodule over $\C$, we have that $(\tensor*[^{x}]{c}{}f)(\tilde{m}_{\g})=f(\tensor*[^{\g x}]{c}{}\tilde{m}_{\g})$, and since $f$ is $A'$-linear, we obtain that
\[(\tensor*[^{x}]{c}{}f)(\tilde{m}_{\g})=\tensor*[^{\g x}]{c}{}f(\tilde{m}_{\g}).\]
According to the previous step, we have that
	\[\Hom{A'}{\tilde{M}}{\tilde{M'}}_1=\set{f\in \Hom{A'}{\tilde{M}}{\tilde{M'}}\mid f(\tilde{M}_{\g})\subseteq \tilde{M'}_{\g}\text{ for all }\g\in\oG}=\Hom{A'\textrm{-}\mathrm{Gr}}{\tilde{M}}{\tilde{M}'}.\]
Because $A'$ is strongly $\oG$-graded, the   categories $B'\Md$ and  $A'\textrm{-}\mathrm{Gr}$  are equivalent. Therefore,
\[\Hom{A'\textrm{-}\mathrm{Gr}}{\tilde{M}}{\tilde{M}'}=\Hom{B'}{M}{M'}.\]
It follows that $\Hom{A'}{\tilde{M}}{\tilde{M'}}$ and $\widetilde{\Hom{B'}{M}{M'}}$ are isomorphic as $\oG$-graded $(A,{A''})$-bimodules over $\C$.
\end{proof}	

\section{\texorpdfstring{$\bar{G}$}{G}-graded Morita equivalences over \texorpdfstring{$\C$}{C}} \label{s:Morita}

\begin{subsec} \label{ss:assmptions}  Let $G'$ be a subgroup of $G$ and $N'$ a normal subgroup of $G'$. We assume that $N'=G'\cap N$ and $G=G'N$, hence  $\oG:=G/N=G'/(G'\cap N)\simeq G'/N'$.

Let $b\in Z(\O N)$ and ${b'}\in Z(\O N')$ be two $\oG$-invariant block idempotents. We denote
\[ A:=b\O G, \qquad A':={b'}\O G',\qquad B:=b\O N, \qquad B': = {b'}\O N'. \]
We know that $A$ and $A'$ are strongly $\oG$-graded algebras, with  1-compo\-nents $B$ and $B'$ respectively.

We also assume that $A$ and $A'$ are  $\oG$-graded algebras over the $\oG$-graded $\oG$-acted $\O$-algebra $\C$,  with structural maps $\zeta:\C\to C_A(B)$ and $\zeta':\C\to C_{A'}(B')$, as in Definition \ref{def:graded_algs_over_C}.

For instance, if we assume, in addition, that $C_G(N)\subseteq G'$, and denote $\bar{C}_G(N):=NC_G(N)/N$, then we may take  $\C:=\O C_G(N)$ (we do this later in this section), which is a strongly $\bar{C}_G(N)$-graded $\bar G$-acted algebra.
\end{subsec}

\begin{defi} Let $\tilde{M}$ be a $\oG$-graded $(A,A')$-bimodule over $\C$. Clearly, the $A$-dual $\tilde{M}^{\ast}=\Hom{A}{\tilde{M}}{A}$ of $\tilde{M}$ is a $\oG$-graded $(A',A)$-bimodule over $\C$. We say that  $\tilde{M}$ induces a $\oG$-\textit{graded Morita equivalence over} $\mathbf{\C}$ between $A$ and $A'$, if $\tilde{M}\otimes_{A'}\tilde{M}^{\ast}\cong A$ as $\oG$-graded $(A,A)$-bimodules over $\C$ and that $\tilde{M}^{\ast}\otimes_{A}\tilde{M}\cong A'$ as $\oG$-graded $(A',A')$-bimodules over $\C$.
\end{defi}

\begin{theorem} \label{th:Morita_extension} Assume that the $(B,B')$-bimodule $M$ and its the $B$-dual $M^{\ast}=\Hom{B}{M}{B}$  induce a Morita equivalence between $B$ and $B'$:
	\[	\xymatrix@C+=3cm{
	B\md \ar@<+.5ex>@{->}[r]^{M^{\ast}\otimes_B -} & \ar@<+.5ex>@{->}[l]^{M\otimes_{B'}-} B'\md	
	}		\]
	If $M$ extends to a $\Delta^{\C}$-module, then we have the following:
	\begin{enumerate}
		\item[{\rm(1)}] $M^{\ast}$ becomes a $\Delta(A'\otimes_{\C} A\op)$-module;
		\item[{\rm(2)}] $\tilde{M}:=(A\otimes_{\C} {A'}\op)\otimes_{\Delta^{\C}} M$ is a $\oG$-graded $(A,A')$-bimodule over $\C$, $\tilde{M}^{\ast}\simeq(A'\otimes_{\C} A\op)\otimes_{\Delta(A'\otimes_{\C} A\op)} M^{\ast}$ as $\oG$-graded $(A',A)$-bimodules over $\C$, and they induce a $\oG$-graded Morita equivalence over $\C$ between $A$ and $A'$.
	\end{enumerate}
\end{theorem}

\begin{proof} (1) Because $M$ is a $\Delta^{\C}$-module, and  $B$ is a $\Delta(A\otimes_{\C} A\op)$-module, by Proposition \ref{lemma:hom} (2) we obtain that $M^{\ast}:=\Hom{B}{M}{B}$ is a $\Delta(A'\otimes_{\C} A\op)$-module.

(2) Given that $M$ is a $\Delta^{\C}$-module and that $M^{\ast}$ is a $\Delta(A'\otimes_{\C} A\op)$-module, by Theorem \ref{th:equivalent_categories}, we obtain that $\tilde{M}:=(A\otimes_{\C} {A'}\op)\otimes_{\Delta^{\C}} M$ is a $\oG$-graded $(A,A')$-bimodule over $\C$ and $\tilde{M}^{\ast}:=(A'\otimes_{\C} A\op)\otimes_{\Delta(A'\otimes_{\C} A\op)} M^{\ast}$ is a $\oG$-graded $(A',A)$-bimodule over $\C$. Moreover, by Proposition \ref{prop:equivalent_functors}, we have that $\tilde{M}\simeq A\otimes_B M$ as $\oG$-graded $(A,A')$-bimodules over $\C$ and that $\tilde{M}^{\ast}\simeq A'\otimes_{B'}M^{\ast}$ as $\oG$-graded $(A',A)$-bimodules over $\C$.
\par Now, because $M$ extends to a $\Delta^{\C}$-module, by the restriction of scalars in the algebra homomorphism from Remark \ref{remark:o_z_c}, we obtain that $M$ is a $\Delta^{\O}$-module. By \cite[Theorem 5.1.2.]{book:Marcus1999} we obtain that $A\otimes_B M$ and $A'\otimes_{B'}M^{\ast}$ give a $\oG$-graded Morita equivalence between $A$ and $A'$, hence $\tilde{M}$ and $\tilde{M}^{\ast}$ give a $\oG$-graded Morita equivalence between $A$ and $A'$, and they are both $\oG$-graded bimodules over $\C$.
\end{proof}

\begin{remark} \label{r:truncateMorita}	$\oG$-graded Morita equivalence over $\C$ can be truncated \cite[Corollary 5.1.4.]{book:Marcus1999}. In our case, consider $\tilde{M}$ a $\oG$-graded $(A,A')$-bimodule over $\C$ and $\tilde{M}^{\ast}$ its $A$-dual. Assume that $\tilde{M}$ and $\tilde{M}^{\ast}$ induce a $\oG$-graded Morita equivalence over $\C$ between $A$ and $A'$. Let $\bar{H}$ be a subgroup of $\oG$. Then $\tilde{M}_{\bar{H}}:=\bigoplus_{\h\in\bar{H}}\tilde{M}_{\h}$ and $\tilde{M}_{\bar{H}}^{\ast}$ induce a $\bar{H}$-graded Morita equivalence over $\C_{\bar{H}}$ between $A_{\bar{H}}$ and $A'_{\bar{H}}$.
\end{remark}

\begin{subsec} \label{ss:notE(U)} If $U$ is a $B$-module, we denote by
\[E(U):=\mathrm{End}_A(A\otimes_{B}U)\op\]
the $\oG$-graded endomorphism algebra of the $A$-module induced from $U$.
\end{subsec}

\begin{prop} \label{prop:commutative_diagram} Assume that $\tilde{M}$ induces a $\oG$-graded  Morita equivalence over $\C$ between $A$ and $A'$.
Let $U$ be a  $B$-module and let $U'=M^{\ast}\otimes_B U$ be the $B'$-module corresponding to $U$ under the given equivalence.
Then there is a commutative diagram:
\[\xymatrix@C+=3cm{
		E(U) \ar@{->}[r]^{\sim} & E(U')\\
		C_A(B) \ar@{->}[r]^{\sim} \ar@{->}[u] & C_{A'}(B') \ar@{->}[u]\\
	    \C  \ar@{->}[u] \ar@{=}[r]^{\mathrm{id}_{\C}} & \C. \ar@{->}[u]
}\]
\end{prop}
\begin{proof}
According to \cite[Proposition 3.4.]{article:MM2019}, the upper square of this diagram is commutative, where between $C_A(B)$ and $C_{A'}(B')$, we have the following construction:	By  Theorem  Morita II \cite[Theorem 12.12]{book:Faith1973}, we can choose the $\Delta^{\O}$-linear $\oG$-graded bimodule isomorphism:
\[\tilde{\varphi}:\tilde{M}^{\ast}\otimes_{A}\tilde{M}\to A',\]
and by its surjectivity, we may choose the finite set $J$ and the elements $m_j^{\ast}\in M^{\ast}$ and $m_j\in M$, for all $j\in J$ such that:
\[\tilde\varphi(\sum_{j\in J}{m_j^{\ast}\otimes_{B}m_j})=1_{B'}=1_{A'}.\]
	We define $\varphi_2:C_A(B)\to C_{A'}(B')$ as follows:
	\[\varphi_2(c)=\tilde{\varphi}(\sum\limits_{j\in J}{m_j^{\ast}c\otimes_{B}m_j}),\]
	for all $c\in C_A(B)$.
	\par In order to prove that the lower square of the diagram is commutative, we consider an element $c\in \C$ and we prove that $\varphi_2\para{\zeta\para{c}}=\zeta'\para{\mathrm{id}_{\C}\para{c}} \Leftrightarrow \varphi_2\para{\zeta\para{c}}=\zeta'\para{c}$. We have:
	\[\begin{array}{rcl}
		\varphi_2\para{\zeta\para{c}}& = & \tilde{\varphi}(\sum\limits_{j\in J}{m_j^{\ast}\zeta\para{c}\otimes_{B}m_j}),
\end{array}\]
but $m_j^{\ast}\in M^{\ast}=\tilde{M}^{\ast}_1$	and $\tilde{M}^{\ast}$ is a $\oG$-graded $(A',A)$-bimodule over $\C$, so:
	\[\begin{array}{rcl}
		\varphi_2\para{\zeta\para{c}}& = & \tilde{\varphi}(\sum\limits_{j\in J}{^{1}\zeta'\para{c}m_j^{\ast}\otimes_{B}m_j})\\
		& = & \tilde{\varphi}(\sum\limits_{j\in J}{\zeta'\para{c}m_j^{\ast}\otimes_{B}m_j})\\
		& = & \tilde{\varphi}(\zeta'\para{c}\sum\limits_{j\in J}{m_j^{\ast}\otimes_{B}m_j}),
\end{array}\]
but $\zeta'\para{c}\in C_{A'}(B')\subseteq A'$ and $\tilde{\varphi}$ is an $(A',A')$-bimodule isomorphism, thus:
\[\begin{array}{rcl}
		\varphi_2\para{\zeta\para{c}}& = & \zeta'\para{c}\tilde{\varphi}(\sum\limits_{j\in J}{m_j^{\ast}\otimes_{B}m_j})\\
		& = & \zeta'\para{c}.
\end{array}\]
Thus the statement is proved.
\end{proof}

\begin{subsec} We will give a version for Morita equivalences over $\mathcal{C}$ of the so-called ``butterfly theorem" \cite[Theorem 2.16]{ch:Spath2018}, based on \cite[Theorem 4.3.]{article:MM2019}. We start by adapting \cite[Proposition 4.2.]{article:MM2019}. We assume that $C_G(N)\subseteq G'$, and we denote $\bar{C}_G(N):=NC_G(N)/N$ and $\C:=\O C_G(N)$. We consider the algebras:
\[\xymatrix@C+=3cm{
	A:=b\O G & A':={b'}\O G'\\
	C:=b\O NC_G(N) \ar@{-}[r]_{\sim} \ar@{-}[u] & C':={b'}\O N'C_G(N) \ar@{-}[u]\\
	B:=b\O N \ar@{-}[r]^{\tensor*[_{B}^{}]{M}{_{B'}^{}}}_{\sim} \ar@{-}[u] & B':={b'}\O N'. \ar@{-}[u]
}\]
\end{subsec}

\begin{prop} \label{prop:extension_to_C} Let $\C=\O C_G(N)$, and assume that:
\begin{enumerate}
	\item[{\rm(1)}] $C_G(N)\subseteq G'$;
	\item[{\rm(2)}] $M$ induces a Morita equivalence between $B$ and $B'$;
	\item[{\rm(3)}] $zm=mz$, for all $m\in M$ and $z\in Z(N)$.
\end{enumerate}
Then there is a $\bar{C}_G(N)$-graded Morita equivalence between $C$ and $C'$ over $\C$, induced by the $\bar{C}_G(N)$-graded $(C,C')$-bimodule over $\C$:
\[C\otimes_{B}M\simeq M\otimes_{B'}C'\simeq (C\otimes_{\C} C'^{\mathrm{\op}})\otimes_{\Delta(C\otimes_{\C}C'^{\mathrm{\op}})}M.\]
\end{prop}

\begin{proof} It is straightforward that $\C:=\O C_G(N)$ is a $\bar{C}_G(N)$-graded $\bar{C}_G(N)$-acted algebra and that there exist $\bar{C}_G(N)$-graded $\bar{C}_G(N)$-acted homomorphisms $\zeta:\C\to C_C(B)$ and $\zeta':\C\to C_{C'}(B')$.
	
Now, we prove that there is a $\bar{C}_G(N)$-graded Morita equivalence over $\C$ between $C$ and $C'$.  It suffices to prove that $C\otimes_{B}M$ is actually a $\bar{C}_G(N)$-graded $(C,C')$-bimodule over $\C$.
	
By \cite[Proposition 4.2.]{article:MM2019}, we know that $C\otimes_{B}M$ is actually a $\bar{C}_G(N)$-graded $(C,C')$-bimodule.
Moreover, observe that for all $\g\in\oG$, $a_{\g}\in C_{\g}$, $m\in M$ and $c\in \C$, we have:
\[(a_{\g}\otimes_{B} m)c  =  a_{\g}c\otimes_{B} m =  \tensor*[^{\g}]{c}{}(a_{\g}\otimes_{B} m).\]
Consequently, $C\otimes_{B}M$ is a $\bar{C}_G(N)$-graded $(C,C')$-bimodule over $\C$.
\end{proof}

\begin{theorem}[Butterfly theorem] \label{th:butterfly} Let $\hat G$ be another group with normal subgroup $N$, such that the block $b$ is also $\hat G$-invariant. Let ${\C}=\O C_{{G}}(N)$. Assume that:
\begin{enumerate}
\item[{\rm(1)}] $C_G(N)\subseteq G'$,
\item[{\rm(2)}] $\tilde M$ induces  a $\oG$-graded Morita equivalence over $\C$ between $A$ and $A'$;
\item[{\rm(3)}] the conjugation maps $\varepsilon:G\to \mathrm{Aut}(N)$ and $\hat{\varepsilon}:\hat{G}\to \mathrm{Aut}(N)$ satisfy $\varepsilon(G)=\hat{\varepsilon}(\hat{G})$.
\end{enumerate}
Denote $\hat{G}'=\hat{\varepsilon}^{-1}(\varepsilon(G'))$. Then there is a $\hat G/N$-graded Morita equivalence over $\hat{\C}:=\O C_{\hat{G}}(N)$ between $\hat{A}:=b\O \hat{G}$ and $\hat{A}':={b'}\O \hat{G}'$.
\end{theorem}

\begin{proof}
\newcommand{\T}{\mathcal{T}}
Consider the following diagram:
\[\xymatrix@C+=1cm{
	\hat{A}:=b\O \hat{G} & A:=b\O G \ar@{-}[r]^{\tilde{M}}_{\sim} & A':={b'}\O G' & \hat{A}':={b'}\O \hat{G}'\\
	b\O NC_{\hat{G}}(N) \ar@{-}[u] & b\O NC_{G}(N) \ar@{-}[u] \ar@{-}[r]_{\sim}& {b'}\O N'C_{G}(N) \ar@{-}[u] & {b'}\O N'C_{\hat{G}}(N) \ar@{-}[u]\\
	 &B:=\O N b \ar@{-}[u] \ar@{-}[r]^{M}_{\sim} \ar@{-}[ul]& B':=\O N' {b'}. \ar@{-}[u] \ar@{-}[ur]&
}\]
By the proof of \cite[Theorem 2.16]{ch:Spath2018}, we have that $\hat{G}'\leq \hat{G}$, $C_{\hat{G}}(N)\leq \hat{G}'$, $\hat{G}= N \hat{G}'$ and $N'=N\cap \hat{G}'$.
By \cite[Theorem 4.3.]{article:MM2019} we have that $\hat{A}\otimes_{B} M$ induces a $\hat{G}/N$-graded Morita equivalence between $\hat{A}$ and $\hat{A}'$.
By the proof of Proposition \ref{prop:extension_to_C}, we have that $\hat{A}\otimes_{B} M$ is a $\hat G/N$-graded $(\hat{A},\hat{A}')$-bimodule over $\hat{\C}:=\O C_{\hat{G}}(N)$.
\end{proof}

\begin{subsec}  Condition (1) of Proposition \ref{prop:extension_to_C} may be replaced with the assumption that the $(B,B)$-bimodule $M$ is $\bar G$-invariant, that is, $A_{\bar g}\otimes_B M\otimes_{B'}A'_{{\bar g}^{-1}}\simeq M$ as $(B,B')$-bimodules, for all $\bar g\in\bar G$, to get a slightly stronger statement.

Let $\bar G[b]=G[b]/N$ be the stabilizer of $B$ as a $(B,B)$-bimodule, that is, $\bar G[b]$ is the largest subgroup $\bar H$ of $\bar G$ such that $C_A(B)_{\bar H}$ is a crossed product (see \cite[2.9]{article:Dade1973}). Then $\bar G[b]$ is a normal subgroup of $\bar G$, and in particular, we have that $C_G(N)\subseteq  G[b]$.
\end{subsec}

\begin{corollary} \label{cor:Gb} Assume that the $(B,B)$-bimodule $M$ is $\bar G$-invariant.  Assume also that $zm=mz$, for all $m\in M$ and $z\in Z(N)$. Then $\bar G[b]=\bar G[b']$ {\rm(}hence $C_G(N)\subseteq  G'${\rm)}, and there is a $\bar G[b]$-graded Morita equivalence over $\C:=\O C_G(N)$ between $A_{\bar G[b]}$ and $A'_{\bar G[b]}$.
\end{corollary}

\begin{proof} Indeed, recall that the functor $M^*\otimes_B(-)\otimes_BM$ is an equivalence which sends the $(B,B)$-bimodule $B$ to the $(B',B')$-bimodule $B'$. Since $M$ is $\bar G$-invariant, it also sends $A_{\bar g}$ to $A'_{\bar g}$, for all $\bar g\in \bar G$. This immediately implies that $\bar G[b]\subseteq\bar G[b']$, and by symmetry $\bar G[b]= G[b']$. The rest follows by the proof of Proposition \ref{prop:extension_to_C}.
\end{proof}

\begin{remark} \label{r:Gb} On the other hand, still without condition (1) of Proposition \ref{prop:extension_to_C}, assume that $\tilde M$ induces a  $\bar G$-graded Morita equivalence between $A$ and $A'$ (so in particular, $M$ is $\bar G$-invariant). Then, by \cite[Theorem 5.1.8]{book:Marcus1999}, $C_A(B)\simeq C_{A'}(B')$ and $\bar G[b]=\bar G[b']$. If, in addition, $zm=mz$, for all $m\in M$ and $z\in Z(N)$, then, by Corollary \ref{cor:Gb}, $\tilde M$ is a  $\bar G$-graded $(A,A')$-bimodule over $\C:=\O C_G(N)$ (in fact, even over $C_A(B)_{\bar G[b]}$).
\end{remark}

\section{Scott modules} \label{s:scott}

Koshitani and Lassueur constructed in \cite{article:KL1}  and \cite{article:KL2} Morita equivalences induced by certain Scott modules. We show here that their constructions can be extended to obtain group graded Morita  equivalences over $\C=\O C_G(N)$.

\begin{subsec} Let $Q$ be a Sylow $p$-subgroup of $G$ and let $G'=N_G(Q)$ and $N'=G'\cap N$. In this situation, we have that  $C_G(N)\subseteq G'$, and let $\C=\O C_G(N)$ and $Z=Z(N)$. Denote also
\[K=\{(g,g')\mid \bar g=\bar g' \textrm{ in } \bar G=G/N\simeq G'/N'\}.\]
  Let $b\in Z(\O N)$ and ${b'}\in Z(\O N')$ be the principal block idempotents, and denote
\[ A:=b\O G, \qquad A':={b'}\O G',\qquad B:=b\O N, \qquad B': = {b'}\O N'. \]
As in \cite{article:KL1}  and \cite{article:KL2}, we consider the Scott module $\operatorname{Sc}(N\times N',\Delta Q)$, and we refer to  \cite[Section 4.8]{book:NT1988} for the properties of Scott modules.
\end{subsec}

\begin{prop} With the above notations, assume that $p$ does not divide the order of $\bar G$ and denote $M=\operatorname{Sc}(N\times N',\Delta Q)$. Then
\[M\simeq \operatorname{Res}^K_{N\times N'}\operatorname{Sc}(K,\Delta Q).\]
In particular, $M$ may be regarded as a $\Delta(A\otimes {A'}^{\mathrm{op}})$-module, and moreover, $M$ may be chosen such that $\tilde M:=A\otimes_B M$ is a $\bar G$-graded $(A,A')$-bimodule  over $\mathcal{C}=\O C_G(N)$ between $A$ and $A'$.
\end{prop}

\begin{proof} By the Scott-Alperin theorem \cite[Theorem 4.8.4]{book:NT1988} (which also holds over $\mathcal{O})$, the Green correspondent  of $M$ is an $\mathcal{O}N_{N\times N'}(\Delta Q)$-module on which $\Delta Q$ acts trivially, and may be seen as the projective cover of the trivial module $\mathcal{O}_{N_{N\times N'}(\Delta Q)/\Delta Q}$. Similarly, the Green correspondent of $\operatorname{Sc}(K,\Delta Q)$ is the projective cover of the trivial module $\mathcal{O}_{N_K(\Delta Q)/\Delta Q}$. Since $\bar G$ is a $p'$-group, and since  $\mathcal{O}_{N_{N\times N'}(\Delta Q)/\Delta Q}=\operatorname{Res}\mathcal{O}_{N_K(\Delta Q)/\Delta Q}$, it follows by \cite[Corollary 3.1.11]{book:Marcus1999} that the projective cover of  $\mathcal{O}_{N_{N\times N'}(\Delta Q)/\Delta Q}$ is the restriction to $\mathcal{O}N_{N\times N'}(\Delta Q)/\Delta Q$ of the projective cover of  $\mathcal{O}_{N_K(\Delta Q)/\Delta Q}$. This means that the Green correspondent of $M$ extends to an $\mathcal{O}N_K(\Delta Q)$-module, which is the Green correspondent of $\operatorname{Sc}(K,\Delta Q)$. It follows by \cite[Theorem 6.8]{article:Dade1984} that $M$ extends to an $\mathcal{O}K$-module, which is isomorphic to $\operatorname{Sc}(K,\Delta Q)$.

For the last statement, by Remark \ref{r:Gb}, it is enough to show that $M$ can be chosen such that $zm=mz$ for all $m\in M$ and $z\in Z$. Note that by \cite[Corollary 8.5]{book:NT1988}, we have that $M\simeq \operatorname{Sc}(N\times N',\Delta(ZQ))$. It is easy to see that  the $\mathcal{O}(N\times N')$-module $\operatorname{Ind}_{\Delta(ZQ)}^{N\times N'}$ satisfies
\[(z\otimes 1)(n\otimes n')\otimes_{\mathcal{O}\Delta(ZQ)}1=(1\otimes z^{-1})(n\otimes n')\otimes_{\mathcal{O}\Delta(ZQ)}1\]
for all  $n\in N$, $n'\in N'$ and $z\in Z$. Since $M$ is a direct summand of $\operatorname{Ind}_{\Delta(ZQ)}^{N\times N'}$, the claim follows.
\end{proof}

\section{Rickard equivalences over \texorpdfstring{$\mathcal{C}$}{C}} \label{s:Rickard}

By using the remarks made in \cite[5.2.1]{book:Marcus1999}, we may extend the results of Section \ref{s:Morita} to the case of Rickard equivalences.  We keep the notations and assumptions of \ref{ss:assmptions}, and we use $\mathcal{H}^b$ to denote a bounded homotopy category.

\begin{subsec} Let $\tilde M$ be a bounded complex of $\bar G$-graded $(A,A')$-bimodules, with $1$-component $M$, which is a bounded complex of $\Delta(A\otimes_\mathcal{O}{A'}\op)$-modules. Recall that $\tilde M$ and its dual $\tilde M^*$ induce a $\bar G$-graded Rickard equivalence between $A$ and $A'$, if there are isomorphisms
\[\tilde{\varphi}:\tilde{M}^{\ast}\otimes_{A}\tilde{M}\to A'\qquad \textrm{ and } \qquad \tilde{\psi}:\tilde{M}\otimes_{A'}\tilde{M}^{\ast}\to A,\]
in the appropriate bounded homotopy categories of $\bar G$-graded bimodules.

We say that this equivalence is over $\mathcal{C}$ if $\tilde M$ is a complex of $\bar G$-graded $(A,A')$-bimodules over $\mathcal{C}$.
\end{subsec}

\begin{subsec} Let $U$ be a bounded complex of $B$-modules, and denote
\[E(U)=\mathrm{End}_{\mathcal{H}(A)}(A\otimes_BU)\op\]
This is compatible with the notation in  \ref{ss:notE(U)}, and $E(U)$ is still a $\bar G$-graded algebra.  We also have the $\bar G$-graded algebra map
\[\theta:C_A(B)\to E(U),\] induced by $c\mapsto(a\otimes u\mapsto ac\otimes u)$. In fact $A\otimes_BU$ is a complex of $\bar G$-graded $(A,C)$-bimodules.
\end{subsec}

\begin{subsec} Assume that $\tilde M$  induces a $\bar G$-graded Rickard equivalence between $A$ and $A'$, and sends $U$ to $U'\in\mathcal{H}^b(B')$. Then the functor $M^*\otimes_A(-)\otimes_AM$ is also a Rickard equivalence which preserves gradings by $\bar G\times \bar G$-sets of complexes of bimodules, and sends $A$ to $A'$ (concentrated in degree $0$).

We have the isomorphism of complexes of $\bar G$-graded bimodules
\[\beta:A'\otimes_{B'} M^{\ast}\to M^{\ast}\otimes_{B}A,\qquad  a'_{\bar g}\otimes_{B'} m^{\ast} \mapsto a'_{\bar g}m^*u_{\bar g}^{-1}\otimes u_{\bar g}\]
for all $\bar g\in \bar G$, and this gives the isomorphism
\[ \varphi_1:E(U)\to E(U'), \qquad    f\mapsto (\beta\otimes_{B}id_U)^{-1}\circ(id_{\tilde{M}^{\ast}}\otimes f)\circ (\beta\otimes_{B}id_U)\]
of complexes of $\bar G$-graded bimodules.

By identifying $C_A(B)$ with $\mathrm{End}_{A\otimes B\op}(A)\op$, we still have the isomorphism
\[\varphi_2:C_A(B)\to C_{A'}(B')\] 
of $\bar G$-graded $\bar G$-acted algebras.
\end{subsec}

\begin{subsec} With these observations, it is easy to see that Remark \ref{r:truncateMorita}, Proposition \ref{prop:commutative_diagram} and Proposition \ref{prop:extension_to_C} still hold by replacing ``Morita equivalence'' with ``Rickard equivalence" (and ``modules" with ``bounded complexes of modules"). On the other hand, according to \cite[Theorem 5.2.5]{book:Marcus1999}, in Theorem \ref{th:Morita_extension} we have to assume in addition that $p$ does not divide the order of $\bar G$. Corollary \ref{cor:Gb} and Remark \ref{r:Gb} can also be adapted to this situation.
\end{subsec}

Consequently, we have the following Rickard equivalence variant of  the Butterfly Theorem (Theorem \ref{th:butterfly}).

\begin{theorem} \label{th:Rickard-butterfly} Let $\hat G$ be another group with normal subgroup $N$, such that the block $b$ is also $\hat G$-invariant. Let ${\C}=\O C_{{G}}(N)$. Assume that:
\begin{enumerate}
\item[{\rm(1)}] $C_G(N)\subseteq G'$,
\item[{\rm(2)}] $\tilde M$ induces  a $\oG$-graded Rickard equivalence over $\C$ between $A$ and $A'$;
\item[{\rm(3)}] the conjugation maps $\varepsilon:G\to \mathrm{Aut}(N)$ and $\hat{\varepsilon}:\hat{G}\to \mathrm{Aut}(N)$ satisfy $\varepsilon(G)=\hat{\varepsilon}(\hat{G})$;
\item[{\rm(4)}] $p$ does not divide the order of $\bar G$.
\end{enumerate}
Denote $\hat{G}'=\hat{\varepsilon}^{-1}(\varepsilon(G'))$. Then there is a $\hat G/N$-graded Rickard equivalence over $\hat{\C}:=\O C_{\hat{G}}(N)$ between $\hat{A}:=b\O \hat{G}$ and $\hat{A}':={b'}\O \hat{G}'$.
\end{theorem}

\section{Character triples and module triples} \label{s:triples}

In this section, we show that the relations $\geq$ and $\geq_c$ given in \cite[Definition 2.1.]{ch:Spath2018} and \cite[Definition 2.7.]{ch:Spath2018} are consequences of Rickard equivalences over $\mathcal{C}$.

\begin{subsec} We consider the notations and the assumptions given in \ref{ss:assmptions}. Let $V$ be a  $G$-invariant simple $\K B$-module, and let $V'$ be a  $G$-invariant  simple $\K B'$-module, where $\K B=\K\otimes_{\O}B=(1\otimes b)\K N$ and $\K B'=\K\otimes_{\O}B'=(1\otimes {b'})\K N'$.

Let $\theta\in \text{Irr}_{\K}(B)$ be a $G$-invariant irreducible character associated to $V$, and let $\theta\in \text{Irr}_{\K}(B')$ be a $G'$-invariant irreducible character associated to $V'$. Thus, $(G,N,\theta)$ and $(G',N',\theta')$ are character triples.
\end{subsec}

\begin{defi} \label{def:module_triple}  We say that $\triple{A}{B}{V}$ is a \textit{module triple}, and we will consider its endomorphism algebra
\[E(V):=\text{End}_{\K A}\left(\K A\otimes_{\K B}V\right)\op.\]
Because we assumed that $\K$ contains all the unity roots of order $|G|$ (see section \ref{s:preliminaries}), $E(V)$ is a twisted group algebra of the form $\K_{\alpha}\oG$, with $\alpha\in Z^2(\oG,\K^{\times})$. We know that the class $[\alpha]\in H^2(\oG, \K^{\times})$ depends only on the isomorphism class of $V$, thus, in fact, $[\alpha]$ is determined by $\theta$.
\end{defi}

We recall from \cite{ch:Spath2018} the definition of the relations $\geq$ and $\geq_c$.

\begin{samepage}
\begin{defi} \label{d:chartriplerel} Let $(G,N,\theta)$ and $(G',N',\theta')$ be two character triples.

a) We write $(G,N,\theta)\geq (G',N',\theta')$ if
	\begin{enumerate}
		\item $G=NG'$ and $N'=N\cap G'$;
		\item there exists projective representations $P$ associated with $(G,N,\theta)$ and $P'$ associated with $(G',N',\theta')$, such that their factor sets $\alpha$ and $\alpha'$ respectively, satisfy $\alpha'=\alpha_{H\times H}$.
	\end{enumerate}
	We say that $(G,N,\theta)\geq (G',N',\theta')$ is given by $(P,P')$.

b)  We write $(G,N,\theta)\geq_c (G',N',\theta')$ if
	\begin{enumerate}
		\item $(G,N,\theta)\geq (G',N',\theta')$;
		\item $C_G(N)\subseteq G'$;
		\item there exists $(P,P')$ giving $(G,N,\theta)\geq (G',N',\theta')$ such that for any $c\in C_G(N)$, there exists $\zeta\in\mathbb{C}^{\times}$ with $P(c)=\zeta \mathrm{id}_{\theta(1)}$ and $P'(c)=\zeta \mathrm{id}_{\theta'(1)}$.
	\end{enumerate}
\end{defi}
\end{samepage}

\begin{prop} \label{p:Delta(V)}	Let $\Delta(V):=\Delta(\K G\otimes E(V)\op).$ Then $\Delta(V)$ is isomorphic to $\K _{\mathrm{Inf}^G_{\oG}\alpha} G$ as $\bar G$-graded algebras. In particular, the $\K N$-module structure of $V$ extends to a $\K _{\operatorname{Inf}^G_{\oG}\alpha} G$-module structure.
\end{prop}

\begin{proof} It is well-known that the induced module $\K G\otimes_{\K N}V$ is a $\oG$-graded $(\K G, E(V))$-bimodule and $(\K G\otimes_{\K N}V)_1=V$, hence $V$ is a $\Delta$-module. We may write
\[	\K _{\text{Inf}^G_{\oG}\alpha} G =\{\sum_{g\in G}a_g\widetilde{g}\mid a_g\in\K\}\text{, where }\widetilde{g}\widetilde{h}=\text{Inf}^G_{\oG}\alpha(g,h)\widetilde{gh},\]
and
\[	E(V)=\K _{\alpha} \oG = \{\sum\limits_{\g\in\oG}a_{\g}\g\mid a_{\g}\in\K\}\text{, where }\g\bar{h}=\alpha(\g,\bar{h})\overline{gh}.\]
Hence, an element of degree $\g$ of $\Delta(V)$ can be written as $gn\otimes \g$, where $n\in N$. We will prove now, that the correspondence
\[gn\otimes \g\stackrel{\varepsilon}{\mapsto} \widetilde{gn}\]
	determines a $\oG$-graded algebra isomorphism from $\Delta$ to $\K _{\text{Inf}^G_{\oG}\alpha} G$.
	We check only the multiplicative property, the others being easy to check. Let $g,h\in G$ and $n,m\in N$. We have
\[	\begin{array}{rcl}
		\varepsilon((gn\otimes \g)(hm\otimes \bar{h}))&=&\varepsilon(gnhm\otimes\bar{g}\bar{h})\\
		&=&\varepsilon(ghnm\otimes\alpha(\bar{g},\bar{h})\overline{gh})\\
		&=&\alpha(\bar{g},\bar{h})\varepsilon(ghnm\otimes\overline{gh})\\
		&=&\alpha(\bar{g},\bar{h})\widetilde{ghnm}
	\end{array}		\]
	and
\[	\begin{array}{rcl}
		\varepsilon(gn\otimes \bar{g})\varepsilon(hm\otimes \bar{h})&=&\widetilde{gn}\widetilde{hm}\\
		&=&\text{Inf}^G_{\oG}\alpha(gn,hm)\widetilde{gn hm}\\
		&=&\alpha(\overline{gn},\overline{hm})\widetilde{gn hm}\\
		&=&\alpha(\bar{g},\bar{h})\widetilde{gn hm}\\
	\end{array}		\]
Henceforth, the statement is proven.
\end{proof}

\begin{remark} \label{remark:similar_projective_representations} The $\Delta(V)$-module structure of $V$ gives rise to the projective $\K$-representation of $G$ associated to $\theta$. Two projective representations $P$ and $P'$ are similar  if and only if $E(V)=\K_{\alpha}\oG$ and $E(V')=\K_{\alpha'}\oG$ are isomorphic as $\oG$-graded algebras, or if and only if  $[\alpha]=[\alpha']$ in $H^2(\oG,\K^{\times})$.  This holds if and only if the $\Delta(V)$-module $V$ is isomorphic to the $\Delta(V')$-module $V'$, via the isomorphism as $\Delta(V)\simeq \Delta(V')$ of  $\oG$-graded algebras.
\end{remark}

Now, we can formulate a version in terms of modules of the order relations between character triples.

\begin{defi}  \label{d:modtriplerel} Let $\triple{A}{B}{V}$ and $\triple{A'}{B'}{V'}$ be two module triples.

a)	We write $\triple{A}{B}{V}\geq \triple{A'}{B'}{V'}$ if
\begin{enumerate}
		\item $G=NG'$ and $N'=N\cap G'$;
		\item there exists a $\oG$-graded algebra isomorphism
\[E(V)=\text{End}_{\K A}(\K A\otimes_{\K B}V)\op\to E(V')=\text{End}_{\K A'}(\K A'\otimes_{\K B'}V')\op.\]
	\end{enumerate}

\begin{samepage}
b) We write $(A,B,V)\geq_c (A',B',V')$ if
\begin{enumerate}
	\item $G=G'N$, $N'=N\cap G'$
	\item $C_G(N)\subseteq G'$
	\item we have the following commutative diagram of $\oG$-graded $\K$-algebras:
	\[\xymatrix@C+=3cm{
		E(V) \ar@{->}[r]^{\sim} & E(V')\\
	    \K \C  \ar@{->}[u] \ar@{=}[r]^{\mathrm{id}_{\K \C}} & \K \C, \ar@{->}[u]
}\]
	where $\K \C=\K C_G(N)$ is regarded as a $\oG$-graded $\oG$-acted $\K$-algebra, with 1-component $\K Z(N)$.
\end{enumerate}
\end{samepage}
\end{defi}

We now give a link  between the relation $\geq_c$ for character triples, the relation $\geq_c$ for module triples and Rickard equivalences over $\C$. Recall that since $\mathcal{K}B$ is a semisimple algebra, the indecomposable objects in  $\mathcal{D}^b(\mathcal{K}B)$ are the simple $\mathcal{K}B$-modules regarded as complexes concentrated in some degree $n\in \mathbb{Z}$.

\begin{theorem} \label{t:eq-triples} Assume that  $C_G(N)\subseteq G'$, and that the complex $\tilde{M}$ induces a $\oG$-graded Rickard equivalence over $\C:=\O C_G(N)$ between $A$ and $A'$.

Let $V$ be a $G$-invariant simple $\K B$-module with character $\theta$, and let $V'$ be a $G'$-invariant simple $\K B'$-module corresponding to $V$ via the given correspondence,  with character $\theta'$. Then we have that $(A,B,V)\geq_c (A',B',V')$ and $(G,N,\theta)\geq_c (G',N',\theta')$.
\end{theorem}

\begin{proof} We know by \cite[Proposition 3.3]{article:Broue1992} that $\K \tilde{M}:=\K\otimes_{\O}\tilde{M}$ and $\K \tilde{M}^{\ast}:=\K\otimes_{\O}\tilde{M}^{\ast}$ induce a Rickard equivalence (clearly $\oG$-graded) between $\K A$ and $\K A'$. To prove that $E(V)\simeq E(V')$ as $\bar G$-graded algebras,
it is enough to prove that $\K A'\otimes_{\K B'}V'$ corresponds to $\K A\otimes_{\K B}V$.   Indeed, because $V$ corresponds to $V'$, we have that $V'\cong \K \Ms\otimes_{\K B}V$, where $\Ms$ is the $1$-component of $\tilde{M}^{\ast}$.	Therefore,
\[\begin{array}{rcl}
		\K \tilde{M}^{\ast}\otimes_{\K A}(\K A\otimes_{\K B} V) & \cong & (\K \tilde{M}^{\ast}\otimes_{\K A}\K A)\otimes_{\K B} V\\
		& \cong & \K \tilde{M}^{\ast}\otimes_{\K B} V\\
		& \cong & (\K A'\otimes_{\K B'}\K \Ms)\otimes_{\K B} V\\
		& \cong & \K A'\otimes_{\K B'}(\K \Ms\otimes_{\K B} V)\\
		& \cong & \K A'\otimes_{\K B'}V',\\
	\end{array}	\]
so the statement is proved.

As $\tilde{M}$ and $\tilde{M}^{\ast}$ are complexes of $\oG$-graded bimodules over $\C$, it is clear that $\K\tilde{M}:=\K\otimes_{\O}\tilde{M}$ and $\K\tilde{M}^{\ast}:=\K\otimes_{\O}\tilde{M}^{\ast}$ are complexes of $\oG$-graded $(\K A,\K A')$-bimodules over $\K \C$, and of $\bar G$-graded $(\K A',\K A)$-bimodules over $\K \C$, respectively. By the proof of Proposition \ref{prop:commutative_diagram}, we deduce that  $(A,B,V)\geq_c (A',B',V')$.

Since $E(V)\simeq E(V')$, we have that $\Delta(V)\simeq \Delta(V')$ as $\bar G$-graded algebras. Henceforth, by Proposition \ref{p:Delta(V)} and Remark \ref{remark:similar_projective_representations}, we obtain $(G,N,\theta)\geq (G',N',\theta')$.

Finally, to verify condition (3) in Definition \ref{d:chartriplerel} b), let $c\in C_G(N)$. Since $V$ is a simple $\K N$-module, $c$ acts on $V$ as a scalar $\zeta\in\K$. By the commutativity of the diagram from condition (3) in Definition \ref{d:modtriplerel} b), $c$ acts as the same scalar $\zeta$  on the $\Delta(V')$-module $V'$. This shows that $(G,N,\theta)\ge_c(G',N',\theta')$.
\end{proof}


\end{document}